\documentclass {article}
\usepackage{authblk}
\usepackage{etex}
\usepackage[utf8]{inputenc}
\usepackage[english]{babel}
\usepackage{amsmath}
\usepackage{amsthm}
\usepackage{amssymb}
\usepackage{sectsty}
\usepackage{titlesec}
\usepackage{color}
\usepackage[color,matrix,arrow]{xy}
\usepackage{amsgen}
\usepackage{amstext}
\usepackage{amsbsy}
\usepackage{amsopn}
\usepackage{amsfonts}
\usepackage{eepic}
\usepackage{graphicx}
\usepackage{epsf}
\usepackage{pstricks}
\usepackage{xfrac}
\usepackage{float}
\usepackage{todonotes}
\usepackage{tikz}
\usetikzlibrary{automata, positioning}
\usepackage{enumerate}
\usepackage{eqnarray}
\usepackage{faktor}
\usepackage{ mathrsfs }
\usepackage{enumitem}
\xyoption{all}
\usepackage{calc}

\usepackage{pgf}
\usepackage{tikz-cd}
\usetikzlibrary{automata, arrows.meta, positioning,decorations.pathmorphing}

\newcommand{\footrecall}[1]{
} 
\usetikzlibrary{snakes,shapes,arrows,automata}

\usepackage{tikz}
\usetikzlibrary{arrows.meta, decorations.pathmorphing}

\titleformat*{\section}{\large\bfseries}
\titleformat*{\subsection}{\normalsize \bfseries}

\newcommand{\N}{\mathbb{N}}

\newcommand{\C}{\mathcal{C}}

\newcommand{\Cyc}{\text{Cyc}}
\newcommand{\Geo}{\text{Geo}}
\newcommand{\ConjGeo}{\text{ConjGeo}}
\newcommand{\ConjSL}{\text{ConjSL}}
\newcommand{\ConjMinLenSL}{\text{ConjMinLenSL}}
\newcommand{\QG}{\text{QG}}
\newcommand{\CycGeo}{\text{CycGeo}}
\newcommand{\SL}{\text{SL}}

\newcommand{\Rat}{\text{Rat}}

\newcommand{\CF}{\text{CF}}

\newcommand{\mc}{\mathcal}

\theoremstyle{definition}
\newtheorem{theorem}{Theorem}[section]
\newtheorem{corollary}[theorem]{Corollary}
\newtheorem{definition}[theorem]{Definition}

\newtheorem{proposition}[theorem]{Proposition}

\newtheorem{lemma}[theorem]{Lemma}
\newtheorem{example}[theorem]{Example}
\newtheorem{remark}[theorem]{Remark}

\title{Conjugacy languages and conjugacy growth relative to subsets of groups}

\author[1]{André Carvalho}
\affil[1]{Center of Mathematics of the University of Porto\\ Rua do Campo Alegre, s/n\\ 4169-007 Porto, Portugal\\
\texttt{andrecruzcarvalho@gmail.com}}
\author[2]{Ana-Catarina C. Monteiro}
\affil[2]{Center for Mathematics and Applications (NOVA Math), NOVA School of Science and Technology (NOVA FCT)\\
	\texttt{acatarinacm@gmail.com}}

\date{}

\begin{document}

\maketitle

\begin{abstract}
In this paper, we explore conjugacy languages when the base problem is the generalized conjugacy problem (with constraints): given $g\in G$ and $U\subset G$, does $g$ have a conjugate in $U$ (with conjugators in a certain subset)? To do so, for  subsets $U,V\subseteq G$,  we define the corresponding languages $\ConjGeo(U,V)$,  $\CycGeo(U)$, $\ConjSL(U)$ and $\ConjMinLenSL(U,V)$, following the previously studied cases where $U=V=G$. Our results cover several classes of groups: for free groups, we prove that $\ConjGeo(U,V)$ and $\ConjMinLenSL(U,V)$ are regular if $U$ and $V$ are rational subsets; for hyperbolic groups, we show that if $H$ is a quasiconvex subgroup, then $\ConjGeo(H)$ and $\ConjMinLenSL(H)$ are regular; for virtually cyclic groups, we show that $\ConjSL(U)$ is regular if $U$ is rational; and, for virtually abelian groups, we prove that $\ConjGeo(U)$ belongs to a certain class of languages $\C$ when the language of words representing elements of $U$ also belongs to $\C$.
We also define relative conjugacy growth and show that its behavior can be heavily dependent on the choice of subset, proving that, for the free group, it can either be a polynomial of any degree or exponential.
\end{abstract}

 \section{Introduction}

Given a group $G$, two elements $x,y\in G$ are said to be \emph{conjugate} if there is some $z\in G$ such that $x=z^{-1}yz$, in which case we write $x\sim y$. The \emph{conjugacy problem} CP$(G)$ consists of, given $x,y\in G$, deciding whether $x\sim y$ or not. 
The \emph{membership problem}, MP$(G)$, also known as the \emph{generalized word problem}, consists of, given a finitely generated subgroup $H\leq G$ and an element $x\in G$, deciding whether $x\in H$ or not. This can be considered more generally for subsets belonging to a reasonably well-behaved class instead of subgroups (e.g. rational or context-free subsets). The \emph{generalized conjugacy problem with respect to $\mc C$}, GCP$_\mc C(G)$, where $\mc C$ is a class of subsets of $G$, consists then of, given $x\in G$ and $K\in \mc C$, deciding whether there is some $y\in K$ such that $x\sim y$. Clearly, if $\mc C$ contains all singletons (which occurs if $\mc C$ is the class of rational subsets or the class of cosets of finitely generated subgroups), this is indeed a generalization of the conjugacy problem.

In case $\mc C$ is the class of the rational subsets of $G$, which we denote by $\Rat(G)$, or simply by $Rat$ when the context is clear, the following is easy to see, where $\leq$ means that the problem on the left-hand side is reducible to the one on the right-hand side: 
\begin{align*}
&\text{WP}(G)&&\leq &&&\text{MP}_{\Rat}(G)\\
&\quad\rotatebox{270}{$\leq\;\;$}&&  &&&\rotatebox{270}{$\leq\;\;$}\quad\;\;\\
&\text{CP}(G)&&\leq &&&\text{GCP}_{\Rat}(G)
\end{align*}

Additionally, we will also consider versions of the conjugacy problems with certain \emph{constraints} on the conjugators. In \cite{[LS11]}, it is proved that the \emph{generalized conjugacy problem with rational constraints} with respect to rational subsets of finitely generated virtually free groups is decidable, meaning that, given a virtually free group $G$, there is an algorithm taking as input two rational subsets $L,K\in \Rat(G)$ and an element $x\in G$ and decides if there is some $z\in L$ such that $z^{-1}xz\in K$. We remark  that Diekert, Guti\'errez and Hagenah prove in \cite{[DGH05]} that the existential theory of equations with rational constraints in free groups is PSPACE-complete and, if $K_1,K_2\in \Rat(G)$, the statement that there is an element of $K_1$ conjugate to $g$ by an element of $K_2$ can be expressed as: 
$$\exists z\in K_2\, \exists x\in K_1 \,:\, z^{-1}xz=g.$$ 

For virtually free groups (in fact, for quasi-isometrically embeddable rational subsets of hyperbolic groups), similar conclusions were obtained in \cite{[DG10]}. Further, several results on the generalized conjugacy problem for different classes of groups and languages were proved in \cite{[Car22b]}.

Formal languages have been a core tool to study problems of different natures in group theory. 
Regarding the conjugacy problem, there are connections with  formal languages in different directions. In \cite{[HRR11]}, Holt, Rees and R\"over define the conjugacy problem of a group to be the set of ordered pairs of words $(u,v)$ such that $u$ and $v$ are conjugate in $G$ and show that the conjugacy problem of a finitely generated virtually free group is asynchronously index, while virtually cyclic groups are precisely the finitely generated groups having a synchronously one-counter (equivalently, an asynchronously context-free) conjugacy problem. In \cite{[Lev23]}, Levine showed that virtually free groups where precisely those for which every conjugacy class is a context-free subset, in the sense that, if $c$ is a conjugacy class, then the language of all words representing an element in $c$ is  context-free. Similarly, in \cite{[CNB25]}, the authors show that virtually cyclic groups are precisely those for which conjugacy classes are one-counter subsets (in the same sense). 
 These are natural generalizations of the following well-known theorems about the \textit{word problem}, i.e., the language of all words representing the identity element.
 \begin{theorem}\label{Thm:AnisimovMSHerbst}
Let $G$ be a finitely generated group. Then we have:
\begin{enumerate}
\item  The word problem of $G$ is regular $\iff$ $G$ is finite (Anisimov, \cite{[Ani71]}).
\item  The word problem of $G$ is context-free  $\iff$ $G$ is virtually free (Muller \& Schupp, \cite{[MS83]}).
\item The word problem of $G$ is one-counter  $\iff$ $G$ is virtually cyclic (Herbst, \cite{[Her91]}).
\end{enumerate}
\end{theorem}
 
 In a different direction, \cite{[CHHR16]} introduces several languages consisting of representatives of conjugacy classes in a group, and shows that, when considering languages formed by words representing elements from each conjugacy class (for example, words of minimal length) these languages can exhibit well-behaved properties, even for complex classes of groups. For example, the authors show that the language containing all words of minimal length in each conjugacy classes is regular. Interesting connections between the language-theoretic properties of these languages and the behavior of their conjugacy growth series have been established in \cite{[CHHR16]}.
 
 Regarding the generalized conjugacy problem, in \cite{[LS11]}, Ladra and Silva show that virtually free groups have decidable generalized conjugacy problem with respect to rational subsets. To do so, they prove that, for a virtually free group, an  element $g\in G$ and  a rational subset $U\subseteq G$,  the language of words representing elements that conjugate $g$ into $U$ is regular and effectively constructible. Further, in \cite{[CS24]}, the authors study the doubly generalized conjugacy problem with respect to rational subsets via language-theoretic methods. In fact, when studying these problems, the main objects are not conjugacy classes themselves, but rather unions of the conjugacy classes of elements in a rational subset. For instance, in \cite{[CS24]}, the strategy was to prove that, for a rational subset $U$, 
 the union of all conjugacy classes of elements in $U$ is \emph{in some sense} a context-free subset.

In this context, the goal of this paper is to study conjugacy languages in the sense of \cite{[CHHR16]} for the generalized version of the conjugacy problem (in some cases, with constraints).

First, we introduce the concept of conjugacy growth relative to a subset. It was proved in \cite{[Coo05]} that the cumulative conjugacy growth of a free group is exponential. To highlight the importance of the chosen subset, and the additional complexity arising from the generalized conjugacy problem, we show that, depending on the subset, the relative growth can be a polynomial of any degree. We also prove that the relative conjugacy growth with respect to rational subsets is either polynomial or exponential.

\newtheorem*{thm_crescimento}{Theorem \ref{thm_crescimento}}
\begin{thm_crescimento}
Let $d\in \N$ and $X=\{a,b\}$ be a basis of $F_2$. There exists a rational subset $U_d\in \Rat(F_2)$ such that the growth of $c_{F_2,X,U_d}(n)$ is  polynomial of degree $d-1$ with the associated series being rational. Additonally, there is no rational subset of $F_2$ with intermediate relative conjugacy growth. \end{thm_crescimento}

Now, given $K,L \subseteq G$, define 
$$\alpha(K,L) = \bigcup_{u \in L} u^{-1}Ku.$$
When $L=G$, we simply write $\alpha(K)$ to denote  $\alpha(K,G)$ .
If $X$ is a generating set of $G$, $\tilde X=X\cup X^{-1}$ is a set of monoid generators. Also, as usual, $\pi:\tilde X^*\to G$ will be the canonical surjective homomorphism.
Following \cite{[CHHR16]}, we define, for any subsets $U,V\subseteq G$, the following languages on $\tilde X^*$:
\begin{align*}
\ConjGeo_X(U,V)&=\ConjGeo_X(G)\cap \alpha(U,V)\pi^{-1}\\
\ConjMinLenSL_X(U,V)&=\{w_g\in \Geo_X(\alpha(U,V))\mid |g|=|g|_c\}\\
\ConjSL_X(U)&=\{z_c\in Geo_X(\alpha(U))\mid c \text{ is a conjugacy class}\}.
\end{align*}

To motivate the subsequent work, we present a proposition that connects the \emph{niceness} of $\ConjGeo(U,V)$ for subsets in a given class with the decidability of the generalized conjugacy problem with constraints.

\newtheorem*{prop_mot}{Proposition \ref{prop_mot}}
\begin{prop_mot}
Let $\mc C$ be a class of subsets and $\mc L$ be a class of languages.
	Suppose that the following hold:
	\begin{enumerate}[label=\roman*.]
		\item If $U,V\in \mathcal{C}$, then $\ConjGeo(U,V)\in \mathcal{L}$ and it can be computed;
		\item If $U\in \mc C$ and $g\in G$, then $Ug\in \mc C$  and it can be computed;
		\item $G$ has decidable conjugacy problem;
		\item $\mathcal{L}$ has decidable membership problem.
		\end{enumerate}
		Then, $G$  has a decidable $\mathcal{C}$-generalized conjugacy problem with $\mathcal{C}$-constraints.
\end{prop_mot}
The previous proposition shows that if a group has decidable conjugacy problem, it suffices to show that  if $\mc C$ is a class of subsets such that $\ConjGeo(U)$ is recursive for all $U\in \mc C$, then the group has a decidable $\mc C$-generalized conjugacy problem. We remark that Bodart \cite{[CB25]} provides an example of a group $G$ and a finite generating set $X$ such that $\ConjGeo_X(G)=\ConjGeo_X(G,G)$ is context-free, the word problem is decidable, but the conjugacy problem is not.  

Further, it is proved in \cite{[CS24]} that, for $U,V \in {\rm Rat}\,F_A$, $\alpha(U,V)\pi^{-1}$ is a context-free language. In \cite{[CHHR16]}, it is proved that, if $G$ is a hyperbolic group, then $\ConjGeo(G)$ and $\ConjMinLenSL(G)$ are regular languages, so, for free groups, $\ConjGeo(U,V)=\ConjGeo(G)\cap \alpha(U,V)\pi^{-1}$ is context-free. We show that this language and $\ConjMinLenSL(U,V)$ are, in fact, regular. 

\newtheorem*{ConjGeo_livres}{Theorem \ref{ConjGeo_livres}}
\begin{ConjGeo_livres}
	Let $F_X$ be a free group and $U, V $ be rational subsets of $F_X$. Then $\ConjGeo(U,V)$ and $\ConjMinLenSL(U,V)$ are regular languages.
\end{ConjGeo_livres}

For hyperbolic groups, such a result does not hold. In fact, hyperbolic groups don't have, in general, decidable membership problem, and so the $\Rat$-generalized conjugacy problem is not decidable for hyperbolic groups, from where we deduce  (recall that hyperbolic groups have decidable conjugacy problem) that $\ConjGeo(U)$ does not necessarily belong to a class of languages with decidable membership. In particular, it is not necessarily regular (not even recursive). We show that, if we consider quasiconvex subgroups $H$, where membership is decidable,  the situation changes, and, in this case, $\ConjGeo(H)$ and $\ConjMinLenSL(H)$ are regular. 

\newtheorem*{ConjGeo_hyperb}{Theorem \ref{ConjGeo_hyperb}}
\begin{ConjGeo_hyperb}
If $G$ is a hyperbolic group and $H$ is a quasiconvex subgroup of $G$, then $\ConjGeo(H)$ and $\ConjMinLenSL(H)$ are regular.
\end{ConjGeo_hyperb}

In the case of virtually cyclic groups, we take the analysis a step further. It was proved in \cite{[CS24]} that all rational subsets of virtually cyclic groups (more precisely, of virtually free groups) are the image of a regular language of geodesics. Consequently, for any rational subset $U$ of a virtually cyclic group, it follows directly from the discussion in the hyperbolic case that $\ConjGeo(U)$ and $\ConjMinLenSL(U)$ are regular. Moreover, in this setting, we show that $\ConjSL(U)$ is also regular.

\newtheorem*{thm_virt_cyc}{Theorem \ref{thm_virt_cyc}}
\begin{thm_virt_cyc}
	Let $G$ be a virtually cyclic group. Then for all rational subsets $U$ of $G$, the language $\ConjSL(U)$ is regular.
\end{thm_virt_cyc}

Finally, we turn to virtually abelian groups. Knowing that $\Geo(U)$ (and in particular $\ConjGeo(U)$) do not need to be regular when we consider a rational subset of a virtually abelian group $G$, we focus our generalization on those subsets $U$ for which $\Geo(U)$ belongs to a predefined class of languages. Further,  we assume that $\mc C$ is a full semi-AFL and study the language $\ConjGeo(U)$ for subsets $U \in \mc C^{\forall}(G)$.
\newtheorem*{thm_va}{Theorem \ref{thm_va}}
\begin{thm_va}
	Let $G$ be a virtually abelian group, $N$ be an abelian finite index normal subgroup of $G $, $U$ be a subset of $N$ and $\mc C$ be a class of languages closed for finite unions and the intersection with regular languages. Then there exists a finite generating set $Z$ of $G$ such that if  $\Geo_Z(U)\in \mc C$, then $\ConjGeo_Z(U)$ and $\ConjMinLenSL_Z(U)$ are both languages in $\mc C$.
\end{thm_va}

\newtheorem*{virtab_final}{Theorem \ref{virtab_final}}
\begin{virtab_final}
		Let $\mathcal{C}$ be a class of languages that is full semi-AFL and $G$ a virtually abelian group. We have that if $U$ is a subset of $G$ in $\mc C^{\bullet}(G)$, then $\alpha(U)$ is also a subset in $\mc C^{\bullet}(G)$.
\end{virtab_final}

The paper is organized as follows. In Section 2, we introduce the concepts and results that are essential for the work developed in this article. Section 3 deals with the growth of the number of conjugacy classes relative to subsets and establishes a link between relative conjugacy languages and decidability of the generalized conjugacy problem. Sections 4, 5, 6, and 7 are dedicated to proving language-theoretic properties in free, hyperbolic, virtually cyclic, and virtually abelian groups, respectively. The paper ends with some open questions, which may serve as directions for future research.

\section{Preliminaries}
\subsection{General Concepts}
\label{quasigeodesics preliminaries}

In this paper, all groups considered will be finitely generated and all generating sets will be finite.

Let $G$ be a group, $X$ be a set of generators, $\tilde X=X\cup X^{-1}$ be a set of monoid generators, and $\pi:\tilde X^*\to G$ be the canonical surjective homomorphism. For any word $w\in \tilde X^*$, denote by $l(w)$ the length of $w$ over $X$ and by $w^{[i]}$ the prefix of $w$ with $i$ letters (if $i>l(w)$, then $w^{[i]}=w$).

 For any group element $g$, we define the \textit{length of $g$} with respect to $X$, denoted by $|g|_X$, to be the length of the shortest representative word for $g$ over $X$. When the generating set is clear from context, we will simply write $|g|$. 
  
 The conjugacy class of $g$ in $G$ will be denoted by $[g]$ and we define $|g|_c$, \textit{the minimum length up to conjugacy,} by
$$
|g|_c :=\min\{|h|: h\in [g]\}.
$$
We will often write  $g^h$ to represent $h^{-1}gh$.

A word $w\in \tilde X^*$ is said to be a \textit{geodesic} if $l(w)=|w\pi|$ and a \textit{cyclic geodesic} (or \textit{fully reduced}) if $w$ and all its cyclic permutations are geodesics. We denote the set of all geodesic words over $X$ by $\Geo_X(G)$. For any $w_1,w_2\in X^*$, we will write $w_1=_G w_2$ to say that $w_1\pi=w_2\pi$, i.e., $w_1$ and $w_2$ represent the same group element.

For $\lambda\geq 1$ and $\varepsilon \geq 0$, we define $w$ to be \textit{$(\lambda,\
\varepsilon)$-quasigeodesic} if 
$$
\frac{1}{\lambda}(j-i)  - \varepsilon \leq  d(w^{[i]}\pi,w^{[j]}\pi) \leq \lambda (j-i) +\varepsilon,
$$
for all $0\leq i\leq j \leq l(w)$.
A $(\lambda,\varepsilon)$-quasigeodesic is a \emph{fully $(\lambda,\varepsilon)$-quasireduced word} if $w$ and all of its cyclic permutations are $(\lambda,\
\varepsilon)$-quasigeodesic words.

Further, we fix a total order on the finite set $X$ and consider the shortlex ordering, $<_{sl}$ of $\tilde X^*$. For each $g\in G$, $w_g$ will represent the smallest word in $\tilde{X}^*$ (with respect to  the shortlex ordering) that represents the element $g\in G$. For each conjugacy class $c$, we define the  \textit{shortlex conjugacy normal form} of $c$ to be the shortlex least word $z_c$ over $X$ representing an element of $c$, that is, $z_c\pi\in c$ and $z_c\leq_{sl} w$, for all $w\in X^*$ with $w\pi \in c$. This notation will be kept throughout the paper.

Finally, for any group $G$ and generating set $X$, let $\Gamma_X(G)$ denote the Cayley graph of $G$ with respect to $X$. When the group and generating set under consideration are clear from the context, we will simply write $\Gamma$.

\subsection{Formal languages and group theory}

A word in $\tilde X^*$ is said to be reduced if it has no subwords of the form $xx^{-1}$ or $x^{-1}x$ for any $x\in X$.
We say that a language $L$ is \textit{reduced} if all words in $L$ are reduced and we will denote by $\Cyc(L)$ the set of all cyclic permutations of elements of $L$.  Further, $L$ is \textit{piecewise testable} if it is defined by a regular expression that combines terms of the form $\tilde{X}^*x_1\tilde{X}^*x_2\cdots \tilde{X}^*x_k\tilde{X}^*$ using Boolean operations of union, intersection and complementation, where $k\geq0$  and $x_i\in X$, for every $i\in [k]$.

Let $\mathcal{C}$ be a class of languages. We say that $\mathcal{C}$ is a \textit{cone} if $\mathcal{C}$ is closed under homomorphisms, inverse homomorphisms and intersections with regular languages. If $\mathcal{C}$ is a cone and closed for unions, we say that $\mathcal{C}$ is a \textit{full semi-AFL}.

Beyond the classification of languages, we turn to concepts concerning subsets of groups, where linguistic tools offer insight into their algebraic structure. 
While the study of finitely generated subgroups is a central theme in group theory, arbitrary subsets are too unstructured to yield meaningful results. However, when subsets are defined via appropriate language-theoretic conditions, rich and informative algebraic properties can be obtained.

Let $G$ be a group. Following the notation in \cite{[Her91],[CNB25]}, we write, for any $S\subseteq G$:
\begin{equation*}
	\begin{aligned}
		S\in \mathcal{C}^{\forall}(G) &\text{  if  } S\pi^{-1} \in \mathcal{C}, \\
		S\in \mathcal{C}^{\exists}(G) &\text{  if  } \exists L\in \mathcal{C}:\ L\pi =S,
	\end{aligned}
\end{equation*}

and $\mathcal{C}^{\bullet}$ to represent either $\mathcal{C}^{\forall}$ or $\mathcal{C}^{\exists}$.

We now present three results that will be crucial in Section \ref{virt_abelian}.

\begin{proposition}{\cite[Lemma 4.1]{[Her91]}} \label{prod_rac}
	Let $\mathcal{C}$ be a cone, $G$ be a finitely generated group and $K$ be a rational subset of $G$. Then, we have:
	$$
	L\in \mathcal{C}^{\bullet}(G) \implies LK,\ KL\in \mathcal{C}^{\bullet}(G).
	$$
\end{proposition}

For any group $G$ and any finite index subgroup $H$ of $G$, we define
$$
\mathcal{C}^{\bullet}(G\mid H)=\{K\in \mathcal{C}^{\bullet}(G)\mid K\subseteq H\}.
$$
\begin{theorem}{\cite[Theorem 2.9]{[CNB25]}} \label{prop_restricao}
	Let $G$ be a finitely generated group and $\mathcal{C}$ be a full semi-AFL. Then, for all finite index subgroups $H\leq_{f.i.} G$, we have
	$$
	\mathcal{C}^{\bullet}(G\mid H)=\mathcal{C}^{\bullet}(H).
	$$
\end{theorem}

\begin{theorem}{\cite[Theorem 2.11]{[CNB25]}} \label{prop_fit_ger}
	Let $\mathcal{C}$ be a full semi-AFL. Let $G$ be a finitely generated group, and let $H\leq_{f.i.} G$. Let $\{b_1,\cdots ,b_n\}$ be a right transversal of $H$ in G. Then
	$$
	\mathcal{C}^{\bullet}(G)=\Big\{\bigcup_{i=1}^n L_ib_i \mid L_i\in \mathcal{C}^{\bullet}(H), \forall i\in \{1,\cdots ,n\} \Big\}.
	$$
\end{theorem}

A subset $K\subseteq G$ is said to be \emph{rational} if there is some regular language $L\subseteq \tilde X^*$ such that $L\pi=K$. Therefore, if we denote by $Reg$ the class of all regular languages, we have that $K$ is rational if $K\in Reg^{\exists}(G)$.   
We will denote by $\Rat(G)$ 
the class of all rational 
subsets of $G$.
This will be the sole exception to the use of the $\mc C^\bullet$ notation. Rational subsets are particularly important, as they generalize the notion of finitely generated subgroups.

\begin{theorem}{\cite[Theorem III.2.7]{[Ber79]}}	\label{AnisimovSeifert}
Let $H$ be a subgroup of a group $G$. Then $H\in \Rat(G)$ if and only if $H$ is finitely generated.
\end{theorem}

For a finite set $X$ and any language $L\subseteq \tilde{X}^*$ we define 
$$
\overline{L}:= \{w\in \tilde{X}^*\mid w \text{ is reduced and there exists } u\in L \text{ such that } u\pi=w\pi\}.
$$
Naturally, given a subset $U\subseteq F_X$ a subset of the free group $F_X$, we define
$$
\overline{U}:=\overline{U\pi^{-1}},
$$
i.e., the set of all reduced words that represent elements of $U$.

For any subsets $U,V\subseteq F_X$, we say that $UV$ is \textit{reduced} if $\overline{U}\,\overline{V}$ is a reduced language.

In the free group case, Benois's Theorem provides us with a useful characterization of rational subsets in terms of reduced words representing the elements in the subset.
\begin{theorem}[Benois, \cite{[Ben79]}] \label{benois}
	Let $F_X$ be a finitely generated free group and let $U\subseteq F_X$ be a subset of the free group. Then $U$ is a rational subset of $F_X$ if and only if $\overline{U}$ is a regular language on $X^*$
\end{theorem}

We now proceed to define two important classes of groups: automatic and biautomatic groups.

Let $G$ be a group with generating set $X$. An \emph{automatic structure} of $G$ with respect to $X$ is a collection of finite-state automata consisting of a word acceptor and a family of multiplier automata, one for each $x \in X \cup \{1\}$. The word acceptor accepts at least one word in $X^*$ representing each element of $G$, and for each fixed $x \in X \cup \{1\}$, the corresponding multiplier automaton accepts precisely the pairs $(w_1, w_2)$ of words accepted by the word acceptor such that $w_1 x =_G w_2$. A group is said to be \emph{automatic} if it admits an automatic structure.

Automatic groups can also be characterized by groups admitting a word acceptor whose language satisfies the \emph{fellow traveller property}: if  $L$ is the language accepted by the word acceptor automaton, there exists a constant $C \in \mathbb{N}$ such that for all $u, v \in L$ with $(u, v)$ accepted by one of the multiplier automata, we have that
\[
\forall n \in \mathbb{N}, \quad d(u^{[n]}, v^{[n]}) \leq C.
\]
Equivalently, when reading, synchronously, words that end at distance at most 1 the distance between the corresponding prefixes remains uniformly bounded.

If $G$ admits both \emph{left multiplier automata} (automata recognizing left multiplication by each generator in $X$) and \emph{right multiplier automata} (automata recognizing right multiplication by each generator in $X$), then $G$ admits a \emph{biautomatic structure}, and in this case $G$ is said to be \emph{biautomatic}.

From a geometric point of view, biautomaticity can be also described in terms of the fellow travel property: if $L$ is the language accepted by the word acceptor automaton, i.e., $L\pi = G$, then there exists a constant $C\geq 0$ such that  for all $u, v \in L$ starting or ending at distance at most 1, we have
	\[
	\forall n \in \mathbb{N}, \quad d(u^{[n]}, v^{[n]}) \leq C.
	\]

\subsection{Conjugacy Languages}

In \cite{[CHHR16]}, Ciobanu et al.\ consider different languages related to conjugacy in finitely generated groups. Namely, we define:

\begin{align*}
	\ConjGeo_X(G)&=\{w\in X^*\mid \ell(w)=\min\{\ell(u): u\pi\in [w\pi]\}\}=\{w\in \tilde X^*\mid \ell(w)=|w\pi|_c\}\\
	\CycGeo_X(G)&=\{w\in X^*\mid \text{ $w$ is a cyclic geodesic}\}\\
	\SL_X(G)&=\{w_g\in X^*\mid g\in G\} \\
	\ConjMinLenSL_X(G)&=\{w_g\in X^*\mid |g|=|g|_c\}\\
	\ConjSL_X(G)&=\{z_c\in X^*\mid c \text{ is a conjugacy class}\}.
\end{align*}
Observe that all of the defined languages are contained in $\Geo(G)$.

Regarding these languages, several results were established in \cite{[CHHR16]}, characterizing them based on the underlying group structure. We present three of these results below, which were the starting point for some of the contributions of this paper.

\begin{theorem}\cite[Theorem 3.1]{[CHHR16]}
	Let $G\langle X\rangle$ be a word hyperbolic group. Then $\ConjGeo(G)$ and $\ConjMinLenSL(G)$ are regular.
\end{theorem}

\begin{theorem}\cite[Theorem 3.3]{[CHHR16]}
	Let $G$ be a virtually abelian group. There exists a finite generating set $Z$ for $G$ such that $\ConjGeo(G)$ is piecewise testable. Furthermore, there is an ordering of $Z$ with respect to which $G$ is shortlex automatic and $\ConjMinLenSL_Z(G)$ is regular.
\end{theorem}

\begin{theorem}\cite[Theorem 4.2]{[CHHR16]}
	Let $G$ be a virtually cyclic group. Then for all generating sets of $G$ the set of shortlex conjugacy normal forms $\ConjSL$ is regular.
\end{theorem}

We denote the class of context-free languages by $\CF$. In \cite{[CS24]}, Carvalho and Silva prove the following theorem regarding the set $\alpha(K,L)$ for rational subsets $K,L$ of a free group $F_X$. 
\begin{theorem}\cite[Theorem 3.5]{[CS24]}
	Let $K,L\in \Rat(F_X)$. Then $\alpha(K,L)\in \CF^\forall(F_X)$.
\end{theorem}

\section{Conjugacy  relative to subsets}
In \cite{[LS11]}, Ladra and Silva show that virtually free groups have decidable generalized conjugacy problem with respect to rational subsets, meaning the problem of taking as input an element $g$ and a rational subset $U$ of a finitely generated virtually free group $G$, and deciding whether $g$ has a conjugate in $U$ or not. This problem has also been studied for different classes of groups in \cite{[Car23b]}.  In \cite{[CS24]}, the authors study the doubly generalized conjugacy problem with respect to rational subsets via language-theoretic methods.  When studying these problems, the main objects are not conjugacy classes themselves, but rather unions of the conjugacy classes of elements in a rational subset. 

Beyond decidability, there are important notions related to conjugacy in a group. Two examples of such notions are conjugacy growth and conjugacy languages. In this section, we will adapt these notions to the corresponding equivalent when, instead of conjugacy, we consider generalized conjugacy.

\subsection{Relative conjugacy growth}

In free groups, the study of the growth, as a function of $n$, of the number of conjugacy classes represented by words of length up to $n$ is equivalent to studying the number of cyclically reduced words of length less or equal to $n$, as each conjugacy class contains exactly one cyclically reduced word, up to cyclic permutation.

The \emph{strict conjugacy growth function} is defined as 
$$c_{G,X}(n)=\#\{[g]\mid |g|_c=n\},$$
while the  \emph{cumulative conjugacy growth function} is defined as 
$$cc_{G,X}(n)=\#\{[g]\mid |g|_c\leq n\}.$$

Several studies have explored the relation between the  growth function and the cumulative conjugacy growth function, providing examples of groups where these functions exhibit similar behavior, as well as cases where their growth differs substantially. For further details, see \cite{[GS10]}.

Since the number of cyclically reduced words  of length $n$ grows exponentially with $n$ in a free group, then $c_{F_r,X}(n)$ (and  $cc_{F_r,X}(n)$) is (asymptotically) exponential. For a more detailed analysis see \cite{[Coo05],[CK02],[AC17]}.

We define the \textit{conjugacy growth function relative to a subset} $U$ as
$$c_{G,X, U}(n)=\#\{[g]\mid |g|_c=n \wedge [g]\cap U\neq \emptyset\},$$ and its\textit{ cumulative counterpart} as
$$cc_{G,X, U}(n)=\#\{[g]\mid |g|_c\leq n \wedge [g]\cap U\neq \emptyset\}.$$ 

The main goal of this paper is to study \emph{relative conjugacy languages}, that is languages of representatives of elements in conjugacy classes containing elements in a rational subset $U$, generalizing previous known results for the case $U=G$. This is the analogous notion of conjugacy languages when, instead of the conjugacy problem, we are working in the generalized conjugacy problem.
To illustrate the impact of the choice of the subset in the language of representatives, we start by showing that, in a finitely generated free group, the \emph{(cumulative) conjugacy growth relative to a rational subset $U$} in the free group can be a polynomial of any degree, while the (cumulative) conjugacy growth (which corresponds to the case $U=G$) is exponential. We also prove that these are the only two possible options for the asymptotic behavior of the conjugacy growth. We thank Corentin Bodart for his comments that significantly improved the previous version of the following theorem, as well as its proof.

\begin{theorem}\label{thm_crescimento}
Let $d\in \N$ and $X=\{a,b\}$ be a basis of $F_2$. There exists a rational subset $U_d\in \Rat(F_2)$ such that the growth of $c_{F_2,X,U_d}(n)$ is  polynomial of degree $d-1$ with the associated series being rational. Additonally, there is no rational subset of $F_2$ with intermediate relative conjugacy growth.
\end{theorem}
\begin{proof}		
	Let  $L_d=\{w\in X^*\mid w \text{ contains exactly $d$ copies of $b$}\}.$ The language $L_d$ is regular as it can be recognized by the following finite-state automaton:
	
	
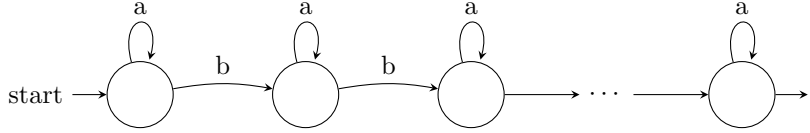
\begin{center}
	\begin{tikzpicture}[>=stealth, on grid, auto]
		\node[state, initial] (q0) {};
		\node[state, right=2.2cm of q0] (q1) {};
		\node[state, right=2.2cm of q1] (q2) {};
		\node[right=1.8cm of q2] (dots) {$\cdots$};
		\node[state, right=1.8cm of dots] (qn) {};
		
		\path[->]
		(q0) edge[loop above] node {a} ()
		edge[bend left=10] node {b} (q1)
		(q1) edge[loop above] node {a} ()
		edge[bend left=10] node {b} (q2)
		(q2) edge[loop above] node {a} ();
		
		\draw[->] (q2) -- (dots);
		\draw[->] (dots) -- (qn);
		
		\path[->] (qn) edge[loop above] node {a} ();
		\node[right=1cm of qn] (out) {};
		\draw[->] (qn) -- (out);
		
	\end{tikzpicture}
	\end{center}
	
	We have that $|L_d\cap \widetilde X^n|= \binom{n}{d}=\Theta(n^d)$. Denote
	\begin{align*}
		U_d&=L_d\pi, \\
		U_d(n)&= (L_d\cap \widetilde X^n)\pi.
	\end{align*}
	Notice that all words in $L_d$ are positive, and so cyclicaly reduced. Hence, for all $g,h\in U_d$
	\begin{align*}
		|g|_c&=|g|,\\
		g&\sim h \text{ if and only if $g$ and $h$ are cyclic permutations}.
	\end{align*}
	
	Let $C_n$ be the cyclic group of order $n$ generated by $t$. We can consider the action of $C_n$ on $U_d(n)$ defined as follows: for each $t^k \in C_n$, with $0\leq k\leq n-1$, and $w \in U_d(n)$, the element $w \cdot t^k$ is obtained by cyclically permuting the last $k$ elements of $w$.
	Therefore, by the Burnside's lemma we have

	\begin{align*}
		c_{F_2,X,U_d}(n)&= |U_d(n)/C_n| =\frac{1}{|C_n|}|\{(g,t^k)\in U_d(n)\times C_n\mid g\cdot t^k=g\}|\\
		&= \frac{1}{n}|\{(g,t^k)\in U_d(n)\times C_n\mid g\cdot t^k=g\}|\\
		&= \frac{1}{n}\sum_{s\mid gcd(d,n)}\binom{n/s}{d/s}\cdot s, \\
	\end{align*}
	where the last equality follows from the fact that if \(g \cdot t^k = g\), then \(g\) must be invariant under a rotation by \(k\) positions. This implies that \(g\) can be divided into \(s=\frac{n}{k}\) identical blocks (so \(s\) must divide \(n\)), and each block must contain the same number of \(b\)'s (so \(s\) must also divide \(d\)). Therefore, the words that are fixed by non-trivial cyclic permutations are exactly the following: for each common divisor $s$ of $n$ and $d$, we count the number of words that can be split into $s$ identical blocks. In this case, each block has $n/s$ letters, containing exactly $d/s$ occurrences of $b$. Thus, there are $\binom{n/s}{d/s}$ ways to arrange the $b$'s within each block. Finally, any word that can be decomposed into $s$ identical blocks is invariant under exactly $s$ distinct cyclic permutations, corresponding to the \(s\) possible cyclic shifts of these blocks.

	Now, the main contribution comes from $s=1$, which gives the term $\frac{1}{n}\binom{n}{d}$. For the other terms with $s \ge 2$, we have 
	$\binom{n/s}{d/s} = \frac{1}{(d/s)!} \left( \left(\frac{n}{s}\right)^{d/s} + \text{lower-order terms} \right)$,
	 so dividing by $n$ gives a contribution of order $O(n^{d/2 -1})$,
	and so we get
	$$
	\frac{1}{n}\sum_{s\mid gcd(d,n)}\binom{n/s}{d/s}\cdot s = \frac{1}{n}\binom{n}{d} + O(n^{\frac{d}{2}-1}).
	$$
	Finally, expanding the binomial $\binom{n}{d}$ as a polynomial in $n$, we have
	\[
	\binom{n}{d} = \frac{n^d}{d!} - \frac{d(d-1)}{2d!} n^{d-1} + \dots,
	\]
	so dividing by $n$ gives $\frac{1}{d!} n^{d-1} + O(n^{d-2})$, which absorbs the previous lower-order contributions, which implies
	\[
	\frac{1}{n}\binom{n}{d} + O(n^{\frac{d}{2}-1}) = \frac{1}{d!} n^{d-1} + O(n^{d-2}).
	\]

	Hence, $c_{F_,X,U_d}(n)=\frac{1}{d!}n^{d-1}+O(n^{d-2})$ oscillates between finitely many different polynomials of degree $d-1$ (depending on $gcd(d,n)$, which influences the terms obtained from the previous sum) and so the associated series is rational, since it is defined through quasipolynomials (see, for example,  \cite{[L03]}).
	
	Regarding the final statement of the theorem, let $U$ be a rational subset of $F_2$ and $L$ a regular language such that $L\pi=U$. By Benois's Theorem (Theorem \ref{benois}) we can assume that $L$ is reduced.
	
	On one hand, we have that $\Cyc(L)$ and consequentely $\overline{\Cyc(L)}$ are regular languages.

	On the other hand, by defining $C_R$ as the language of all cyclically reduced words, we have that $C_R$ is regular. In fact, if Red is the (regular) language of reduced words in $\tilde X^*$, then
	$$
	C_R=\text{Red}\setminus  \left (\bigcup_{x\in \tilde X} x^{-1}\text{Red}\, x\right).
	$$
	
	Hence, we can consider the regular language
	$$
	L':= \overline{\Cyc(L)}\cap C_R,
	$$
	and observe that $L'$ consists precisely of the cyclically reduced words representing elements that are conjugate to elements of $L\pi$.
	 Since, up to cyclic permutation, each conjugacy class contains exactly one cyclically reduced word and it has minimal length in that conjugacy class, counting the number of conjugacy classes relative to $U$ with minimal length $n$ corresponds to counting the number of words in $L'$ of length $n$ divided by $n$, ensuring that we do not count separately words that are cyclic permutations of each other, i.e., if $g_{F_2,L'\pi}$ represents the growth in $F_2$ relative to $L'\pi$, then
	$$
	c_{F_2,X,U}(n)=\frac{1}{n}g_{F_2,L'\pi}(n).
	$$
	
	Finally, since regular language only have polynomial or exponential growth \cite{[tro81]}, and the growth of $L'\pi$ corresponds to the growth of $L'$, we get that the growth of $L'\pi$ is either polynomial or exponential.

	Hence $c_{F_2,X,U}$ has either polynimial or exponential growth.   
\end{proof}

\begin{remark}
A similar result, regarding the existence of rational subsets with polynomial conjugacy growth relative to a subset of any degree, could be obtained by considering the cumulative conjugacy growth function, where, by defining the same subsets, we would obtain the following (also by Burnside’s lemma):
		\begin{align*}
		cc_{F_,X,U_d}(n)&= \sum_{i=1}^{n} |U_d(i)/C_i| =\sum_{i=1}^{n}\frac{1}{|C_i|}|\{(g,t)\in U_d(i)\times C_i\mid g^t=g\}|\\
		&= \sum_{i=1}^{n}\frac{1}{i}|\{(g,t)\in U_d(i)\times C_i\mid g^t=g\}|\\
		&= \sum_{i=1}^{n}\Bigl[\frac{1}{i}\sum_{s\mid gcd(d,i)}\binom{i/s}{d/s}\cdot s\Bigr] \\
		&= \sum_{i=1}^{n}\Bigl[\frac{1}{i}\binom{i}{d} + O(i^{\frac{d}{2}-1})\Bigr]= \sum_{i=1}^{n} \Bigl[\frac{1}{d!}i^{d-1}+O(i^{d-2})\Bigr] \\
		&= \frac{1}{d!}n^{d-1}+O(n^{d-2}).
	\end{align*}
	which has also order $d-1$.
\end{remark}

\subsection{Relative conjugacy languages and decidability}

In this paper, we introduce relative conjugacy languages as the analogous concepts when considering conjugacy relative to a subset (with constraints on the conjugators). Fix some total order on $\tilde X$ and recall the notations introduced in Section \ref{quasigeodesics preliminaries}.
We define, for any subsets $U,V\subseteq G$, the following languages on $\tilde X^*$:
\begin{align*}
\Geo_X(U)&=\{w\in \Geo_X(G)\mid w\pi\in U\}\\
\ConjGeo_X(U,V)&=\ConjGeo_X(G)\cap \alpha(U,V)\pi^{-1}\\
\CycGeo_X(U)&=\{w\in \Geo_X(U)\mid \text{ $w$ is a cyclic geodesic}\}\\
 (\lambda,\varepsilon)-\QG_X(U)&=\{w\in U\pi^{-1}\mid \text{ $w$ is a $(\lambda,\varepsilon)$-quasigeodesic word}\} \\
\SL_X(G)&=\{w_g\in \Geo_X(G)\mid g\in G\} \\
\ConjMinLenSL_X(U,V)&=\{w_g\in \Geo_X(\alpha(U,V))\mid |g|=|g|_c\}\\
\ConjSL_X(U)&=\{z_c\in Geo_X(\alpha(U))\mid c \text{ is a conjugacy class}\}.
\end{align*}
When we are considering relative conjugacy without constraints, we  write $\ConjGeo_X(U)$ to denote $\ConjGeo_X(U,G)$. 
 Notice that we always have $\ConjGeo_X(K)\subseteq \CycGeo_X(\alpha(K))$. Moreover, when the generating set is clear from the context, we will omit explicit reference to it when mentioning each of the languages defined above.\\

We will now see how establishing that certain conjugacy languages belongs to a \emph{good} class of languages helps in deciding the generalized conjugacy problem with constraints, provided that the simple version of the conjugacy problem is decidable. We remark that, in general, establishing that $\ConjGeo(G)$ belongs to a class of languages with decidable membership does not suffice to show that the conjugacy problem is decidable. Indeed, Bodart \cite{[CB25]} provides an example of a group $G$ and a finite generating set $X$ such that $\ConjGeo_X(G)$ is context-free, the word problem is decidable, but the conjugacy problem is not. 

\begin{proposition} \label{prop_mot}
Let $\mc C$ be a class of subsets and $\mc L$ be a class of languages.
	Suppose that the following hold:
	\begin{enumerate}[label=\roman*.]
		\item If $U,V\in \mathcal{C}$, then $\ConjGeo(U,V)\in \mathcal{L}$ and it can be computed;
		\item If $U\in \mc C$ and $g\in G$, then $Ug\in \mc C$  and it can be computed;
		\item $G$ has decidable conjugacy problem;
		\item $\mathcal{L}$ has decidable membership problem.
		\end{enumerate}
		Then, $G$  has a decidable $\mathcal{C}$-generalized conjugacy problem with $\mathcal{C}$-constraints.
\end{proposition}
\begin{proof}
		Given $w\in G$ and $U,V \in \mathcal{C}$, we want to decide if there exists $u\in U$ and $v\in V$ such that $w=v^{-1}uv$. 
		We begin by enumerating all words of length smaller than $|w|$ and checking if they represent an element conjugate to $w$ using the decidability of the conjugacy problem, until we find a minimal element $\tilde w$ in the conjugacy class of $w$ (we might have $\tilde w =w$). We then compute a conjugator $z\in G$ such that   $w=z^{-1}\tilde{w}z$. This can be done by enumerating all possible conjugators and testing the equality using the decidability of the word problem (which follows from that of the conjugacy problem).  We know that $Vz^{-1}\in \mc C$ by ii. and we can compute $\ConjGeo(U,Vz^{-1})$, which must be a language in $\mc L$, by i.

		Now, we simply have to decide if $\tilde w \in  \ConjGeo(U,Vz^{-1})$. Indeed, if $\tilde w \in  \ConjGeo(U,Vz^{-1})$, then there exist $u\in U$ and $v\in V$ such that $\tilde{w}=(vz^{-1})^{-1}uvz^{-1}$, which implies that $w=z^{-1}\tilde{w}z=v^{-1}uv$, as intended. Conversely, if there exist $u\in U$ and $v\in V$ such that $w=v^{-1}uv$, then $z^{-1}\tilde{w}z=v^{-1}uv$, and so $\tilde{w}=zv^{-1}uvz^{-1}$. Now, $\tilde{w}\in \alpha(U,Vz^{-1})\pi^{-1}$ and $\tilde{w}\in \ConjGeo(G)$ by construction, which implies that $\tilde{w}\in \ConjGeo(U,Vz^{-1})$.
\end{proof}
	
We remark that condition ii. is not very restrictive. For instance, it follows from \cite[Lemma 4.1]{[Her91]} (see also \cite[Lemma 2.3]{[CNB25]}) that it holds if $\mc C$ is the class of rational or recognizable subsets, or any other defined in a similar way replacing the class of regular languages by any other cone, i.e, a class of languages closed under morphism, inverse morphism and intersection with regular languages.

Dealing with the problem without constraints, that is, putting $V=G$, we immediately obtain the following corollary:
\begin{corollary}\label{corolario decidibilidade}
Let $\mc C$ be a class of subsets and $\mc L$ be a class of languages.
	Suppose that the following hold:
	\begin{enumerate}[label=\roman*.]
		\item If $U\in \mathcal{C}$, then $\ConjGeo(U)\in \mathcal{L}$;
		\item $G$ has a decidable conjugacy problem;
		\item $\mathcal{K}$ has a decidable membership problem.
	\end{enumerate}
	Then, $G$ has a decidable $\mathcal{C}$-generalized conjugacy problem.
\end{corollary}

\begin{remark}
In \cite{[CS24]}, the authors  study the \emph{doubly generalized conjugacy problem}, which consists of deciding, taking as input two subsets $U,V$, deciding whether there are elements $x\in U$ and $y\in V$ such that $x\sim y$. If, beyond decidable membership in the class of languages $\mc C$, we can decide if two languages in $\mc C$ intersect nontrivially, we can also solve the doubly generalized conjugacy problem, as it amounts to deciding if $\ConjGeo(U)\cap \ConjGeo(V)=\emptyset$ or not. Despite being a much stronger condition to impose on the class of languages, in most of the cases covered in this paper we always have regular $\ConjGeo$.
\end{remark}

\section{Free groups}
The main goal of this section is to prove that given two rational subsets $U,V$ of a free group $F_X$, then the languages $\ConjGeo(U,V)$ and $\ConjMinLenSL(U,V)$ are regular. 

Let $\mathcal{A}=(Q,q_0,T,E)$ be a finite automaton recognizing a language $L$ on $\tilde{X}^*$. For any $p,q\in Q$ we define the set $L_{p,q}$ of all words that can be obtained in $\mathcal{A}$ between the states $p$ and $q$.
Notice that these languages are regular (just change the initial and terminal states of $\mathcal{A}$). Further, given subsets of states $S_1, S_2\subseteq Q$, we define $L_{S_1,S_2}$ to be the set of all words that can be read in $\mathcal{A}$ between a state in $S_1$ and a state in $S_2$. To simplify the notation, when $S_i=\{p\}$ for some $i=1,2$ we just write $p$ in the index, instead of $\{p\}$.

\begin{remark}
	Recall that for any $U,V\subseteq F_X$, we define
	$$\ConjGeo(U,V):=\ConjGeo(F_X)\cap \alpha(U,V)\pi^{-1}.$$
	However, when thinking about free groups, since $\ConjGeo(F_X)\subseteq\Geo(F_X)$, we have that
	$$
	\ConjGeo(U,V)= \ConjGeo(F_X) \cap \overline{\alpha(U,V)},
	$$ 
	and so, to simplify the notation, in this section we will write 
	$$
	\ConjGeo(U,V)= \ConjGeo(F_X) \cap \alpha(U,V),
	$$
	i.e., we will be thinking about the subset of the free group $\alpha(U,V)$ as the language of reduced words representing elements in $\alpha(U,V)$, that is, a subset of the free monoid.
\end{remark}

\begin{definition}\label{def_P_K.L}
	Let $F_X$ be a free group and $K, L$ be regular languages on $\tilde X^*$. We define 
	$P_{K,L}$ as the set of all \emph{permutations of words in $K$ by words in $L$}, i.e.,
	$$
	P_{K,L}=\{u\ell\mid \ell\in L,\ \ell u\in K\}.
	$$
\end{definition}
In order to prove the main goal of this section, we start with two auxiliary results. In what follows it is important to have in mind that $\ConjGeo(F_X)$ is regular, for any free group $F_X$ (see \cite{[CHHR16]}).

\begin{proposition} \label{ConjGeo_livres1}
	Let $F_X$ be a free group and $U, V$ be rational subsets of $F_X$ such that $UV$ is reduced. Then $\ConjGeo(U,V)$ is a regular language.
\end{proposition}
\begin{proof}
	Notice that $\overline U$ and $\overline V$ are regular by Theorem \ref{benois}.
	 We start by showing that $P_{\overline{U},\overline{V}}$ (as in Definition \ref{def_P_K.L}) is regular.
	Let  $\mathcal{A}=(Q,q_0,T, E)$ be the minimal automaton recognizing $\overline{U}$.
	
	We claim that
	$$
	P_{\overline{U},\overline{V}} =\bigcup_{p\in Q,\ t\in T}L_{p,t}(L_{q_0,p}\cap \overline{V}).
	$$
		In fact, if $w\in P_{\overline{U},\overline{V}}$, then $w=u\ell$ with $\ell\in \overline{V}$ and $\ell u\in \overline{U}$. Since $\ell u\in \overline{U}$, there exist states $p\in Q$ and $t\in T$ such that $\ell\in L_{q_0,p}$ (and so $\ell\in L_{q_0,p}\cap \overline{V}$) and $u\in L_{p,t}$, which implies $w\in \bigcup_{p\in Q,\ t\in T}L_{p,t}(L_{q_0,p}\cap \overline{V})$.

		On the other hand, if $w\in \bigcup_{p\in Q,\ t\in T}L_{p,t}(L_{q_0,p}\cap \overline{V})$, then $w=u\ell$, for $u$ such that $u\in L_{p,t}$ and $\ell\in L_{q_0,p}\cap \overline{V}$, for some $p\in Q$ and $t\in T$. Then $\ell\in \overline{V}$ and $\ell u\in L_{q_0,p}L_{p,t}\subseteq \overline{U}$. Hence $ w\in P_{\overline{U},\overline{V}}$.
		
		Next, observe that $	P_{\overline{U},\overline{V}}$ is the finite union of regular languages, and so $	P_{\overline{U},\overline{V}} $ is regular.
		
		Now, since $UV$ is reduced, for all $k\in {U}$ and $\ell\in {V}$, $|k^\ell|\geq |k|$, and we have equality if and only if $k=\ell u$, for some $u\in \tilde{X}^*$, i.e., if $\ell$ is a prefix of $k$.

We will now show that 
		$$
			\ConjGeo(U,V)=\ConjGeo(F_X)\cap P_{\overline{U},\overline{V}},
		$$
		and that suffices as $\ConjGeo(F_X)$ and $P_{\overline{U},\overline{V}}$ are regular and regular languages are closed under intersection.
 
		If $w\in \ConjGeo(U,V)$, then by definition, $w\in \ConjGeo(F_X)\cap \alpha(U,V)$. As $w\in \alpha(U,V)$, we get that $w=_G\ell^{-1}k\ell$, for some $\ell\in \overline{V}$ and $k\in \overline{U}$. Since $w\in \ConjGeo(F_X)$, we conclude, from previous observations, that $\ell$ is a prefix of $k$, and so $w\in P_{\overline{U},\overline{V}}$.

		For the other inclusion, we just need to observe that $P_{\overline{U},\overline{V}}\subseteq \alpha(U,V)$. In fact, if $w\in P_{\overline{U},\overline{V}}$, then $w=u\ell$ such that $\ell\in \overline{V}$ and $\ell u\in \overline{U}$. Then $w=\ell^{-1}\cdot \ell u \cdot \ell$, which implies that $w\in\alpha(U,V)$.
\end{proof}

\begin{corollary}\label{ConjGeo_livres2}
	Let $F_X$ be a free group and $U, V$ be rational subsets on $F_X$ such that $V^{-1}U$ is reduced. Then $\ConjGeo(U,V)$ is a regular language.
\end{corollary}
\begin{proof}
	It is well known that if $\overline{V}$ is regular, then $\overline{V}^{-1}$ is also a regular language. Hence, by defining 
	$$
	P'_{\overline{U},\overline{V}}=\{l^{-1}u\mid l^{-1}\in \overline{V}^{-1},\ ul^{-1}\in \overline{U}\},
	$$
	and proceeding as in the proof of Proposition \ref{ConjGeo_livres1}, we get that
	$$
	P'_{\overline{U},\overline{V}} =\bigcup_{p\in Q,\ t\in T}(L_{p,t}\cap \overline{V}^{-1})L_{q_0,p}.
	$$
	
	We can then conclude that $P'_{\overline{U},\overline{V}}$ is regular and, as in Proposition \ref{ConjGeo_livres1}, that 
	$$
	\ConjGeo(U,V)=\ConjGeo(F_X)\cap P'_{\overline{U},\overline{V}},
	$$
which implies that $\ConjGeo(U,V)$ is regular.
\end{proof}

\begin{theorem}\label{ConjGeo_livres}
	Let $F_X$ be a free group and $U, V $ be rational subsets of $F_X$. Then $\ConjGeo(U,V)$ and $\ConjMinLenSL(U,V)$ are regular languages.
\end{theorem}
\begin{proof}
	We follow the strategy of the proof of \cite[Theorem 3.5]{[CS24]}. Let $\mathcal{A}=(Q, q_0, T,E)$ and $\mathcal{A}'=(Q',q_0', T',E')$ denote respectively the minimal automata recognizing $\overline{U}$ and $\overline{V}$. Consider the sets 
	$$
	S=\{(q,p',q')\in Q\times Q'\times Q'\mid L_{q_0q}\cap L'_{q_0'p'}\cap (L'_{q'T'})^{-1},\ L'_{p'q'}\setminus \{1\}\neq \emptyset\},
	$$
and,	for each $a\in \tilde{X}$,
	$$
	Y_a=\{\overline{w^{-1}v_2w}\mid \exists (q,p',q')\in S, v_2\in L'_{p'q'}\cap \tilde{X}^*a,\ w\in L_{qT}\setminus a^{-1}\tilde{X}^*\}\subseteq F_X,
	$$
	and 
		$$
	Z_a=\{\overline{w^{-1}v_2w}\mid \exists (q,p',q')\in S, v_2\in L'_{p'q'}\cap a\tilde{X}^*,\ w\in L_{qT}\setminus a\tilde{X}^*\}\subseteq F_X.
	$$
	
	 It is proved in \cite[Theorem 3.5]{[CS24]} that
	\begin{align*}
		Y_a&=\bigcup_{(q,p',q')\in S}\alpha(L'_{p',q'}\cap\tilde X^*a, L_ {qT}\setminus a^{-1}\tilde X^*), \\
			Z_a&=\bigcup_{(q,p',q')\in S}\alpha(L'_{p',q'}\cap a\tilde X^*, L_ {qT}\setminus a\tilde X^*),
	\end{align*}
	and
	\begin{equation}\label{alpha_prova}
		\alpha(U,V)=\bigcup_{a\in \tilde{X}} (Y_a\cup Z_a).
	\end{equation}
	
	In our case we have the following:
	\begin{itemize}
		\item Since $(L'_{p',q'}\cap \tilde X^*a)( L_ {qT}\setminus a^{-1}\tilde X^*)$ is reduced for all $(q,p',q')\in S$, then $\ConjGeo(L'_{p',q'}\cap \tilde X^*a, L_ {qT}\setminus a^{-1}\tilde X^*)$ is regular by Proposition \ref{ConjGeo_livres1}. Hence the following language is regular (as it is a finite union of regular languages)
		$$
		\tilde{Y}_a:= \bigcup_{(q,p'q')\in S} \ConjGeo(L'_{p',q'}\cap \tilde X^*a, L_ {qT}\setminus a^{-1}\tilde X^*).
		$$
		
		\item Since $( L_ {qT}\setminus a\tilde X^*)^{-1}(L'_{p',q'}\cap a\tilde X^*)$ is reduced for all $(q,p',q')\in S$, then $\ConjGeo(L'_{p',q'}\cap a\tilde X^*, L_ {qT}\setminus a\tilde X^*)$ is regular by Corollary \ref{ConjGeo_livres2}. Hence the following language is also regular:
		$$
		\tilde{Z}_a:= \bigcup_{(q,p'q')\in S} \ConjGeo(L'_{p',q'}\cap a\tilde X^*, L_ {qT}\setminus a\tilde X^*)
		$$
	\end{itemize} 
	
	Therefore, it suffices to prove that
	$$
	\ConjGeo(U,V)=\bigcup_{a\in \tilde{X}} (\tilde{Y}_a\cup \tilde{Z}_a),
	$$
	to show that $\ConjGeo(U,V)$ is regular.
	
	We have the following 
	\begin{align*}
		\tilde{Y}_a:&= \bigcup_{(q,p'q')\in S} \bigl[\ConjGeo(F_X)\cap \alpha(L'_{p',q'}\cap \tilde X^*a, L_ {qT}\setminus a^{-1}\tilde X^*)\bigr]
		\\
		&= \ConjGeo(F_X)\cap \bigl[\bigcup_{(q,p'q')\in S} \alpha(L'_{p',q'}\cap \tilde X^*a, L_ {qT}\setminus a^{-1}\tilde X^*)\bigr]\\
		&= \ConjGeo(F_X)\cap Y_a
	\end{align*}
	Analogously,
	$$
	\tilde{Z}_a=\ConjGeo(F_X)\cap Z_a
	$$
	
	Then
	\begin{align*}
		\bigcup_{a\in \tilde{A}}(\tilde{Y}_a\cup \tilde{Z}_a)&= \bigcup_{a\in \tilde{A}}\bigl[ (\ConjGeo(F_X)\cap Y_a) \cup (\ConjGeo(F_X)\cap Z_a) \bigr] \\
		&= \bigcup_{a\in \tilde{A}}\bigl[ \ConjGeo(F_X) \cap (Y_a\cup Z_a) \bigr] \\
		&= \ConjGeo(F_X)\cap \left[\bigcup_{a\in \tilde{A}}(Y_a\cup Z_a) \right] \\
		&= \ConjGeo(F_X)\cap \alpha(U,V) = \ConjGeo(U,V),
	\end{align*}
	where the second-to-last equality follows from (\ref{alpha_prova}).
	Therefore, $\ConjGeo(U,V)$ is regular, as intended.
	
	Finally, $\ConjMinLenSL(U,V)$ is regular since $\ConjMinLenSL(U,V)=\ConjGeo(U,V)\cap \SL(F_X)$, and $SL(F_X)$ is regular by \cite[Theorem 2.5.1]{[ECHLPT92]}.
\end{proof}

\section{Hyperbolic groups}\label{chp_hyp}

A group is $\delta$-hyperbolic if its Cayley graph satisfies the property that every geodesic triangle is $\delta$-thin, meaning that each side is contained in the $\delta$-neighborhood of the union of the other two sides. 

Let $G$ be a $\delta$-hyperbolic group with generating set $X$, and let $d$ denote the word metric on $G$. Let $u \in X^*$ be a geodesic word and $v \in X^*$ a $(\lambda,\varepsilon)$-quasigeodesic word representing the same element of $G$. We say that $u$ and $v$ satisfy the \emph{asynchronous fellow travel property} if there exists a non-decreasing function $h \colon \{0,1,\dots,|u|\} \to \{0,1,\dots,|v|\}$ with $h(0)=0$ and $h(|u|)=|v|$, and a constant $K \ge 0$ (depending on $\delta,\ \lambda$ and $\varepsilon$), such that, for all $n = 0,1,\dots,|u|$,
\[
d\big(u^{[n]},\, v^{[h(n)]}\big) \le K.
\]
If, in addition, there exists a constant $L \ge 0$ such that 
\[
|h(n+1) - h(n)| \le L \quad \text{for all } n = 0, 1, \dots, |u|-1,
\]
then we say that $u$ and $v$ satisfy the \emph{boundedly asynchronous fellow travel property}.

We now present some useful known results:

\begin{lemma}\cite[Chapter III.$\Gamma$, Lemma 2.9]{[BH99]}\label{bridsonhaefliger}
Let $G$ be a $\delta$-hyperbolic group with respect to a finite generating set $X$. If two fully reduced words $u,v\in F_X$ represent conjugate elements of $G$, then either $\max\{|u|,|v|\}\leq 8\delta+1$ or there exist cyclic permutations $u'$ and $v'$ of $u$ and $v$ and a word $w\in F_X$ of length at most $2\delta+1$ such that $wu'w^{-1}=v'$ in $G$.
\end{lemma}
\begin{lemma}\cite[Lemma 16]{[HRR11]}\label{lema holt qr}
If $u$ and $v$ are fully $(\lambda,\varepsilon)$-quasireduced words representing conjugate elements of a $\delta$-hyperbolic group, then either $\max(|u|,|v|)\leq \lambda(8\delta+2\gamma+\varepsilon+1)$ or there exist cyclic conjugates $u'$ and $v'$ of $u$ and $v$ and a word $\alpha$ with $(\alpha u'\alpha^{-1})\pi=v'\pi$ and $|\alpha|\leq 2(\delta+\gamma)$, where $\gamma$ is the boundedly asynchronous fellow travel constant satisfied by $(\lambda,\varepsilon)$-quasigeodesics with respect to geodesics. 
\end{lemma}

\begin{lemma}\cite[Proposition 3.19]{[CS24]}\label{minepedro}
	If $G$ is a virtually free group with hyperbolicity constant $\delta$, and $K$ is a regular language of geodesics, then there is a regular language $L_K\subseteq \Geo(\alpha(K))$ such that, for every element $g\in K$, there is at least one fully $(1,3\delta+2R+3)$-quasireduced word in $L_K$ representing an element conjugate to $g$, where $R$ is the asynchronously fellow travel constant depending on $(1,3\delta+1,\delta)$.
\end{lemma}

For the hyperbolic case, we prove a similar result, but only for the case where the regular language of geodesics represents a subgroup. We start by remarking that these are precisely the quasiconvex subgroups of $G$. A subgroup $H$ of a hyperbolic group $G$ is \emph{quasiconvex} if there exists a constant \( k \ge 0 \) such that every geodesic in the Cayley graph of \( G \) with endpoints in \( H \) lies within distance \( k \) of \( H \).

\begin{lemma}\label{qcregular}
Let $H$ be a finitely generated subgroup of a hyperbolic group $G=\langle X\rangle$ and $\pi:\widetilde X^*\to G$ be the standard surjective homomorphism. The following are equivalent:
\begin{enumerate}
\item $H$ is quasiconvex;
\item $\Geo(H)$ is regular;
\item there is a regular language of geodesics $L$ such that $L\pi=H$.
\end{enumerate}
\end{lemma}
\begin{proof}
The equivalence between 1 and 3 is proved in \cite[Theorem 2.2]{[GS91]} and it is obvious that 2 implies 3. So, we only have to prove that 3 implies 2.
Consider the language $K=\{(u,v)\mid u,v\in \Geo(G):u\pi=v\pi\}$. Since $\Geo(G)$ is a biautomatic structure of $G$, then, by \cite[Lemma 8.1]{[GS91b]}, $K$ is regular. Hence, the language 
$K\cap (L \times \widetilde X^*)$ is also regular and so is its projection to the second component $K'$. If $w\in \Geo(H)$, then there is some $u\in L$ such that $u\pi=w\pi$, and so $(u,w)\in K$ and $w\in K'$. Moreover, if $w\in K'$, then $w\pi=u\pi$ for some $u\in L\subseteq \Geo(H)$, and $w\in \Geo(H)$. We then have that $\Geo(H)=K'$ is regular.
\end{proof}

\begin{lemma}\label{reg_hyp}
If $G$ is a $\delta$-hyperbolic group with generating set $X$ and $H$ is a quasiconvex subgroup, then there is a regular language $L_H\subseteq \Geo(\alpha(H))$ such that, for every element $g\in H$, there is at least one fully $(1, 3\delta+2R+3)$-quasireduced word in $L_K$ representing an element conjugate to $g$, where $R$ is the asynchronously fellow travel constant depending on $(1,3\delta+1,\delta)$.
\end{lemma}
\begin{proof}
	It is easy to see that, for any $g\in G$, the subgroup $gHg^{-1}$ is also quasiconvex (and so hyperbolic), and so $\Geo(gHg^{-1})$ is regular by Lemma \ref{qcregular}.
	
	Further, since $H$ is quasiconvex, for any geodesic $u:=u_1u_2\cdots u_n$, $u_i\in X$ with end point in the Cayley graph in $H$, there exists $v:=v_1v_2\cdots v_n$ with $v_i\in \tilde X^*$ and $v_i\pi\in H$, such that
	\begin{itemize}
		\item $v_1v_2\cdots v_n=_G u_1u_2\cdots u_n$;
		\item $\exists z_i,\ i\in \{1,\cdots,n-1\}$ $\exists k>0$ such that $u_1=v_1z_1,\ u_n=z_{n-1}^{-1}v_n\ u_i=z_{i-1}^{-1}v_iz_i$, $1<i<n$ and $|z_i|\leq k$, $\forall i$.
	\end{itemize} 
	
	Now, notice that for any cyclic permutation of $u\in \Geo(H)$, we have

	\begin{align*}
			u_{i+1}\cdots u_nu_1\cdots u_i &=_G z_i^{-1}v_{i+1}z_{i+1}\cdots z_{n-1}^{-1}v_nv_1z_1\cdots z_{i-1}^{-1}v_iz_i \\
			&=_G z_i^{-1} \underbrace{v_{i+1}\cdots v_nv_1\cdots v_i}_{\in H}z_i,
	\end{align*}
	i.e., any element in $\Cyc(\Geo(H))$ can be written as a conjugate of a word representing an element of $H$ by a conjugator with bounded length. 
	
	Therefore,
	$$
	\Geo(\Cyc(\Geo(H)))\subseteq \Geo\left(\bigcup_{|z|\leq k} z^{-1}Hz\right)= \bigcup_{|z|\leq k}\Geo(z^{-1}Hz),
	$$
	where the latest language is regular (it is the finite union of regular languages). Also, we observe that for each $h\in H$, by the proof of Lemma \ref{minepedro}, there exists in $\Geo(\Cyc(\Geo(H)))$ a fully $(1, 3\delta+2R+3)$-quasireduced word representing an element conjugate to $h$, and so the same happens with $\bigcup_{|z|\leq k}\Geo(z^{-1}Hz)$. Finally, it is clear that $\bigcup_{|z|\leq k}\Geo(z^{-1}Hz)\subseteq \Geo(\alpha(H))$, and so the proof is concluded.
\end{proof}

The main goal of this section is to show that, within the class of $\delta$-hyperbolic groups, for any quasiconvex subgroup $H$ of $G$, the languages $\ConjGeo(H)$ and $\ConjMinLenSL(H)$ are a regular language of geodesics. This result extends the main theorem of \cite{[CHHR16]} to a more general setting among hyperbolic groups.  We begin by establishing a few auxiliary results.

\begin{lemma}\label{lemma_quasi_geo}
		Let $G=\langle X \rangle$ be a hyperbolic group. Let $u$ be a $(1,r)$-quasigeodesic in $\Gamma_X(G)$ and $a\in X$. Then, $au$ is a $(1,r+2)$-quasigeodesic.
	\end{lemma}
	\begin{proof}
		Let $p$ be the endpoint of $v:=au$ in $\Gamma_X(G)$.
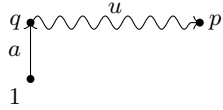
		\begin{center}		
		\begin{tikzpicture}[scale=1.5, every node/.style={font=\small}]
			\coordinate (O) at (0,0);
			\coordinate (A) at (0,0.5);
			\coordinate (P) at (1.5,0.5);
			
			\draw[->] (O) -- (A) node[midway,left] {$a$};
			\draw[->,decorate,decoration={snake,segment length=3mm}] (A) -- (P) node[midway, above] {$u$};
			
			\fill (O) circle (1pt) node[below left] {$1$};
			\fill (A) circle (1pt)  node[left]  {$q$};
			\fill (P) circle (1pt) node[right] {$p$};
		\end{tikzpicture}
	\end{center}
	
	For $n\in \N$, let $p_n:=au^{[n]}$, for all $n\in \mathbb{N}$. We denote by $d_v$ the distance in the word $v$ and by $d$ the distance in the Cayley graph. We have,
	$$
	d_v(1,p_n) = 1 + d_v(q,p_n) \leq 1 + d(q,p_n) + r \leq 1 + r + 1 + d(1,p_n) = d(1,p_n) + r + 2.
	$$ 
	Further, if we take $p_n$ and $p_m$ lying inside $u$, the inequality follows from the fact that $u$ is $(1,r)$-quasigeodesic, i.e.,
	$$d_v(p_m,p_n)\leq d(p_m,p_n)+r\leq d(p_m,p_n)+r+2.$$
	
	Hence, $v$ is a $(1,r+2)$-quasigeodesic.
	\end{proof}

The following lemma is a generalization of \cite[Proposition 3.1]{[Car22c]}.
\begin{proposition}\label{prop_biaut}
	Let $G=\langle X \rangle$ be a $\delta-$hyperbolic group. Let $u$ be a $(1,r)$-quasigeodesic in $\Gamma_X(G)$ and $v$ be a $(1,s)$-quasigeodesic in $\Gamma_X(G)$ with starting points and ending points with distance at most one, respectively. Then, there is a constant $N$ depending on $r,s,\delta$ such that, for all $n\in \mathbb{N}$, we have
	$$
 d_X(u^{[n]}\pi,v^{[n]}\pi)\leq N
	$$
\end{proposition}
\begin{proof} If $u$ and $v$ have the same starting point, this is proved in  \cite[Proposition 3.1]{[Car22c]}.
	
 So, assume that $u$ and $v$ have starting point at distance one. Then, there exists $x\in X$ such that $xu$ and $v$ have the same starting point and end at distance at most one. Since the Cayley graph of $G$ with respect to $X$ is vertex-transitive, we can assume that $1$ is the starting point of both $xu$ and $v$. Let $p$ and $q$ be the endpoints of $xu$ and $v$ and consider a geodesic $w$ from $q$ to $p$.
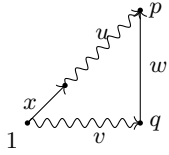
\begin{center}		
	\begin{tikzpicture}[scale=1, every node/.style={font=\small}, decoration={snake,amplitude=0.6mm,segment length=2mm}]
		\coordinate (O) at (0,0);
		\coordinate (A) at (0.5,0.5);
		\coordinate (P) at (1.5,1.5);
		\coordinate (B) at (1.5,0);
		
		\draw[->] (O) -- (A) node[midway,left] {$x$};
		\draw[->,decorate] (A) -- (P) node[midway,above] {$u$};  
		\draw[->,decorate] (O) -- (B) node[midway,below right] {$v$};  
		\draw[->] (B) -- (P) node[midway, right] {$w$};
		
		\fill (O) circle (1pt) node[below left] {$1$};
		\fill (A) circle (1pt);
		\fill (P) circle (1pt) node[right] {$p$};
		\fill (B) circle (1pt) node[right] {$q$};
	\end{tikzpicture}
\end{center}

	By Lemma \ref{lemma_quasi_geo}, $xu$ is a $(1,r')$-quasigeodesic. Let $k=\max\{r',s\}$. Then $xu,v$ and $w$ are $(1,k)$-quasigeodesics. Let $z:=xu$. We have the following 
	
	$$
	d_X(u^{[n]}\pi,v^{[n]}\pi) = d_X(z^{[n+1]}\pi,v^{[n]}\pi) \leq d_X(z^{[n]}\pi,v^{[n]}\pi) + 1 \leq C + 1.
	$$ 	
	
	Where $C$ is a constant obtained as in \cite[Proposition 3.1]{[Car22c]} depending on $r$, $s$ and $\delta$.
\end{proof}

\begin{corollary}
	Let $G=\langle X\rangle$ be a hyperbolic group. Then, for every positive constant $\varepsilon$, the language of $(1,\varepsilon)$-quasigeodesics form a biautomatic structure.
	\end{corollary}
	\begin{proof}
	 It follows from \ref{prop_biaut} and the fact that for any $\epsilon>0$, the language of (1,$\epsilon$)-quasigeodesics is regular (see \cite[Theorem 3.4]{Reg_QG_HYP}).
	\end{proof}

The next result is directly related to \cite[Theorem 3.20]{[CS24]}, which states that in a virtually free group, the set of fully-quasireduced geodesic representatives of conjugates of a rational subset can be captured by a context-free language. In our case, we obtain a regular language of fully reduced representatives for each conjugacy class intersecting a given rational subset that admits a regular language of geodesic representatives.

\begin{lemma}\label{lema_cyc_reg}
If $G$ is a hyperbolic group and $H$ quasiconvex subgroup, then $\CycGeo(\alpha(H))$ is regular.
\end{lemma}
\begin{proof}
	Let $X$ be a finite generating set for $G$, and  $\gamma$ be the boundedly asynchronous fellow travel constant satisfied by  $(1,3\delta+2R+3)$-quasigeodesics with respect to geodesics, where $R$ is the asynchronously fellow travel constant depending on $(1,3\delta+1,\delta)$.
	Let $L_H$ be the language from Lemma \ref{reg_hyp}.
	
	For each $z\in \widetilde X^*$, put 
	$$L_1(z)=\{(u,v)\mid u,v\in (1,\varepsilon)-\QG(G):v\pi=(z^{-1}u z)\pi\},$$
	$$L_2(z)=\{v\in (1,\varepsilon)-\QG(G)\mid \exists u\in \Cyc(L_H):(u,v)\in L_1(z)\},$$
	
	where $\varepsilon=3\delta+2R+3$, and let $S$ be the (finite) language of cyclic geodesics of length at most $(11\delta+2\gamma+2R+4)$ representing conjugates of some element in $H$.

	We claim that 
	\begin{align}\label{formula1}
		\CycGeo(\alpha(H))=S\cup \left(\CycGeo(G)\cap \left( \bigcup_{|z|\leq 2(\delta +\gamma)} \Cyc (L_2(z))\right)\right).
	\end{align}
	Let $w\in \CycGeo(\alpha(H))$. If $\ell(w)\leq (11\delta+2\gamma+2R+4)$, then $w\in S$. Assume then that  $\ell(w)> (11\delta+2\gamma+2R+4)$. Obviously, $w\in \CycGeo(G)$. Moreover, since $w\pi\in \alpha(H)$, then there is an element $g\in H$ such that $w\pi\sim g$. By definition of $L_H$, there is a  fully $(1,\varepsilon)$-quasireduced word $u\in L_H$ such that $u\pi\sim g\sim w\pi$. Since cyclic geodesics are in particular  fully $(1,\varepsilon)$-quasireduced, by Lemma \ref{lema holt qr},  we know that there exist cyclic conjugates $u'$ and $w'$ of $u$ and $w$ and a word $z$ with $(z u'z^{-1})\pi=w'\pi$ and $|z|\leq 2(\delta+\gamma)$, that is, for such $z$, $(u',w')\in L_1(z)$, and thus, $w'\in L_2(z)$, which implies that $w\in \Cyc (L_2(z))$.
	Therefore, $\CycGeo(\alpha(H))\subseteq S\cup \left(\CycGeo(G)\cap \left( \bigcup_{|z|\leq 2(\delta +\gamma)} \Cyc (L_2(z))\right)\right).$

	Now, let $w\in S\cup \left(\CycGeo(G)\cap \left( \bigcup_{|z|\leq 2(\delta +\gamma)} \Cyc (L_2(z))\right)\right).$ If $w\in S$, then $w$ is, by definition of $S$,  a cyclic geodesic representing some element in $\alpha(H)$, that is $w\in \CycGeo(\alpha(H))$. If $w\in \left(\CycGeo(G)\cap \left( \bigcup_{|z|\leq 2(\delta +\gamma)}\Cyc (L_2(z))\right)\right)$, then, to show that $w\in  \CycGeo(\alpha(H))$, it suffices to prove that  $L_2(z)\pi\subseteq \alpha(H),$ 
	for all $z\in \widetilde X^*$ such that $|z|\leq 2(\delta +\gamma)$ 
	(recall that for any $g\in G$, any cyclic permutation of $g$ is in $[g]$). If $w\in L_2(z)$ for such a $z$, then there is $u\in \Cyc(L_H)$ such that $(u,w)\in L_1(z),$ that is, $w\pi=(z^{-1}u z)\pi$, so $w\pi\sim u\pi\in \alpha(H)$, as $\Cyc(L_H)\pi\subseteq \alpha(H)$.
	So, we proved (\ref{formula1}). 
	
	We have that $\CycGeo(G)$ is regular (see \cite{[CHHR16]}) and $S$ is finite, thus regular, so it suffices to show that $\Cyc L_2(z)$ is regular for all $z$ such that $|z|\leq 2(\delta +\gamma)$. By \cite[Lemma 8.1]{[GS91b]}, it follows that $L_1(z)$ is regular, and so is $L_1(z) \cap (\Cyc(L_H) \times \widetilde X^*)$, its projection to the second component, $L_2(z)$ and consequently $\Cyc(L_2(z)) $.
	\end{proof}

\begin{theorem}\label{ConjGeo_hyperb}
If $G$ is a hyperbolic group and $H$ is a quasiconvex subgroup of $G$, then $\ConjGeo(H)$ and $\ConjMinLenSL(H)$ are regular.
\end{theorem}
\begin{proof}
	
	We will adapt the proof from \cite[Theorem 3.1]{[CHHR16]}. Since $\ConjMinLenSL(H)=\ConjGeo(H)\cap \SL(G)$ and $\SL(G)$ is regular, then we only prove the result for $\ConjGeo(H)$.
	
	Let $S$ be the (finite, thus regular) set of words belonging to $\ConjGeo(H)$ of length at most $8\delta +1$. In \cite[Theorem 3.1]{[CHHR16]} we have
	$$\ConjGeo(G)= S \cup\left[(\CycGeo(G) \cap (\bigcup_{k\geq 8\delta +1}X^k))\setminus \left(\Cyc\left(\bigcup_{|\alpha|\leq 2\delta+1} L(\alpha)\right)\right)\right],$$
	and so we get
	\begin{equation*}
		\begin{aligned}
			&\ConjGeo(H)= \ConjGeo(G)\cap \alpha(H)\pi^{-1} \\
			&= [S\cap \alpha(H)\pi^{-1}] \cup \left[(\CycGeo(G) \cap \alpha(H)\pi^{-1} \cap (\bigcup_{k\geq 8\delta +1}X^k))\setminus \left(\Cyc\left(\bigcup_{|\alpha|\leq 2\delta+1} L(\alpha)\right)\right)\right] \\
			&=S \cup\left[(\CycGeo( \alpha(H)) \cap (\bigcup_{k\geq 8\delta +1}X^k))\setminus \left(\Cyc\left(\bigcup_{|\alpha|\leq 2\delta+1} L(\alpha)\right)\right)\right],
		\end{aligned}
	\end{equation*}
	
	where
	$${L(\alpha)=\{v'\in \CycGeo(G)\mid \exists u'\in \CycGeo \text{ such that } \alpha^{-1}u'\alpha=_G v',\ l(v')>l(u')\}.}$$
	We now have the following. From Lemma \ref{bridsonhaefliger}, $\CycGeo(\alpha(H))$ is regular, and so $\CycGeo( \alpha(H)) \cap (\bigcup_{k\geq 8\delta +1}X^k)$ is a regular language. Also, by \cite[Theorem 3.1]{[CHHR16]}, $\bigcup_{|\alpha|\leq 2\delta+1} L(\alpha)$ is regular, which implies that $\Cyc\left(\bigcup_{|\alpha|\leq 2\delta+1} L(\alpha)\right)$ is a regular language.
	
	Hence, $\ConjGeo(H)$ and $\ConjMinLenSL(H)$ are regular languages.
	\end{proof}

Using Lemma \ref{minepedro}, and the fact that, in this setting of virtually free groups, all rational subsets are image of a language of geodesics (\cite[Corollary 3.13]{[CS24]}), by applying the results proved for the hyperbolic case, we can show that for any rational subset $U$, both 
$\ConjGeo(U)$ and $\ConjMinLenSL(U)$ are regular.
So, using Corollary \ref{corolario decidibilidade}, we get another language-theoretic proof of the generalized conjugacy problem for virtually free groups.

 \begin{corollary}
	The generalized conjugacy problem with respect to rational subsets is decidable for virtually free groups.
	\end{corollary}

\section{Virtually cyclic groups}
We now turn to the case of virtually cyclic groups. As already observed in Section \ref{chp_hyp}, for such groups, since, by \cite{[CS24]}, every rational subset is the image of a regular language of geodesics, and consequently, the languages $\ConjGeo(U)$ and $\ConjMinLenSL(U)$ are regular for any rational subset $U$.
 We show here that this regularity extends further: for all generating sets of a virtually cyclic group G, and for all such rational subsets $U \subseteq G$, the language $\ConjSL(U)$ is regular. This strengthens the known result by moving beyond geodesic representatives and confirming regularity at the level of shortlex minimal representatives.

\begin{theorem}\label{thm_virt_cyc}
		Let $G$ be a virtually cyclic group. Then for all rational subsets $U$ of $G$, the language $\ConjSL(U)$ is regular.
\end{theorem}
\begin{proof}
We can assume that $G$ is infinite. Let $H\trianglelefteq G$ of finite index such that $H\cong \mathbb{Z}$. It follows from \cite[Proposition 4.2]{[CHHR16]}, by defining $C$ as the centralizer of $H$ in $G$, that $G\setminus C$ has a finite number of conjugacy classes. Then we have
$$
\ConjSL(U)=(\ConjSL(U)\cap C\pi^{-1}) \cup (\ConjSL(U)\cap (G\setminus C)\pi^{-1}),
$$
where the second term of the union is immediately regular (since is finite).
To show that $\ConjSL(U)\cap C\pi^{-1}$ is regular we start by defining 
$$
V:= \bigcup_{t\in T}(\{u\in \Geo(G)\mid \exists v\in \Geo(G) \text{ s.t. } (u,v)\in L_1(t), v\le_{sl} u\}),
$$
where $L_1(t)$ is as defined by
$$
L_1(t):=\{(u,v)\mid u,v\in \Geo(G),\ v=_G t^{-1}ut\},
$$
which is regular by \cite[Lemma 8.1]{[GS91b]}. By \cite[Proposition 4.2]{[CHHR16]} we have
$$
\ConjSL(G)\cap C\pi^{-1}=[\Geo(G)\setminus V]\cap C\pi^{-1},
$$
and $C\pi^{-1}$ is regular.

Further, we know that $\ConjSL(U)\subseteq \CycGeo(\alpha(U))$, which implies that
\begin{equation*}
	\begin{aligned}
		\ConjSL(U)\cap C\pi^{-1}&=\CycGeo(\alpha(U)) \cap \ConjSL(U)\cap  C\pi^{-1}\\
		&= \CycGeo(\alpha(U)) \cap [\Geo(G)\setminus V]\cap C\pi^{-1} \\
		&= [(\CycGeo(\alpha(U)) \cap \Geo(G))\setminus V]\cap C\pi^{-1} \\
		&= [\CycGeo(\alpha(U))\setminus V]\cap C\pi^{-1}.
	\end{aligned}
\end{equation*}

Now, notice that, as all virtually cyclic groups are virtually free, and that, by \cite[Corollary 3.13]{[CS24]}, $\Geo(U)$ is a regular language, and so, the rational subset $U$ is the imagem of a regular language of geodesics. Then, since all virtually cyclic groups are hyperbolic, by Lemma \ref{lema_cyc_reg}, $\CycGeo(\alpha(U))$ is a regular language and so the same is verified for $\ConjSL(U)\cap C\pi^{-1}$. Consequently we immediately conclude that $\ConjSL(U)$ is regular.
\end{proof}
\section{Virtually abelian groups}\label{virt_abelian}

Building on the results obtained for free groups, hyperbolic groups, and virtually cyclic groups, we now turn our attention to the class of virtually abelian groups. In \cite{[HHR07]} it was shown that for any virtually abelian group $G$, there exists a generating set such that the language of all geodesics is piecewise testable. Later, in \cite[Proposition 3.3]{[CHHR16]}, it was proved that for such a generating set, $\ConjGeo(G)$ is also piecewise testable and $\ConjMinLenSL(G)$ is regular. However, when we move to subsets, the language of geodesics of a rational subset does not even need to be context-free.
\begin{example}
	Consider the free abelian group $G$ generated by $\{a,b,c\}$. Let $K=(abc)^*$, which is a regular language.
	
	We have that $\ConjGeo(K\pi)=\Geo(K\pi)=\text{Perm}((abc)^*)$ and
	$$
	\text{Perm}((abc)^*) \cap a^*b^*c^*=\{a^nb^nc^n\mid n\geq 0\},
	$$
	which is not regular (in fact, it is not even context-free), and so, since context-free languages are closed under intersection with regular languages, we deduce that $\ConjGeo(K\pi)=\Geo(K\pi)$ is not context-free.
\end{example}

Thus, in the case of virtually abelian groups, our approach must be different, and so we consider two different scenarios.  
In the first, we consider the case where $U$ is a subset contained in a finite index abelian subgroup of $G$, and we analyze this case under the assumption that $\Geo(U)$ belongs to a fixed class of languages.  
In the second, we assume that $\mc C$ is a full semi-AFL and study the language $\ConjGeo(U)$ for subsets $U \in \mc C^{\forall}(G)$.

In what follows, in order to simplify notation, for any abelian group $G$ generated by a finite set $X$, any automorphism $\phi$ in $G$ and any $w=x_1\cdots x_n \in X^*$, we will write $\phi(w)=\phi(x_1)\cdots \phi(x_n),$ and so, for any language $L\subseteq X^*$, we will write $\phi(L)$ to represent $\cup_{w\in L}\phi(w)$. Notice that $\phi(X)$ also generates $G$.

\begin{lemma}\label{lema_autom}
	Let $G$ be an abelian group, $U$ a subset of $G$ and $\phi$ an automorphism in $G$. Then $\phi(\Geo_X(U))= \Geo_{\phi(X)}(\phi(U))$ and consequently, for any class of languages $\mc C$ we have $\Geo_X(U)\in \mc C $ if and only if $\Geo_{\phi(X)}(\phi(U))\in \mc C$.
\end{lemma}
\begin{proof}
	Consider $X$ a generating set of $G$ and $x_1,\cdots x_n\in X$ such that $w=x_1\cdots x_n\in \Geo_X(U)$. We have $\phi(w)=\phi(x_1)\cdots \phi(x_n)$. If $\phi(w)\notin \Geo_{\phi(X)}(\phi(U))$, then there exists $y_1,\cdots, y_m\in \phi(X)$ such that $m<n$ and $\phi(w)=_G y_1\cdots y_m$. Then we would have $w=\phi^{-1}\phi(w)=_G\phi^{-1}(y_1)\cdots \phi^{-1}(y_m)$, which contradicts the fact that $w$ is a geodesic. Hence $\phi(\Geo_X(U))\subseteq \Geo_{\phi(X)}(\phi(U))$. The other inclusion is analogous.
	
	For the final equivalence, observe that these automorphisms merely rename the generators while preserving the properties of the languages.
\end{proof}

\begin{theorem}\label{thm_va}
	Let $G$ be a virtually abelian group, $N$ be an abelian finite index normal subgroup of $G $, $U$ be a subset of $N$ and $\mc C$ be a class of languages closed for finite unions and the intersection with regular languages. Then there exists a finite generating set $Z$ of $G$ such that if  $\Geo_Z(U)\in \mc C$, then $\ConjGeo_Z(U)$ and $\ConjMinLenSL_Z(U)$ are both languages in $\mc C$. 
\end{theorem}

\begin{proof}
	Let $A$ be a finite generating set of $N$.
	
	We build the generating set $Z$ of $G$ as in  \cite[Proposition 3.3]{[CHHR16]}. Let $T\cup\{1\}$ be a right transversal of $N$ in $G$, $Y:=T^{\pm 1}$ and $X':=\{x\in N\mid \exists w\in Y^*\ l(w)\leq 4,\ x=_Gw\neq_G 1 \}$. Define $X$ as the closure under inversion and conjugation of the set $A\cup X'$ and $Z:=X\cup Y$.
	We have
	\begin{enumerate}
		\item X$\subseteq$ N,\ Y$\subseteq$ G$\setminus$ N
		\item $X,\ Y$ closed under inversion
		\item $X$ closed under conjugation by elements of $G$
		\item $Y$ contains at least one representative of each nontrivial coset of $N$ in $G$
		\item If $w=_G xy$ with $w\in Y^*,\ l(w)\leq 3$, $x\in N$ and $y\in Y\cup \varepsilon$, then $x\in X$.
	\end{enumerate}
	Regarding property $3$, notice that being closed for conjugation by elements of $G$ is equivalent to saying that $X$ is closed for conjugation by elements of $T$, since for any $g\in G$, we have that $g=nt$ for some $n\in N$ and $t\in T$ and so, for any $x\in X$, since $X\pi\subseteq N$ and $N$ is abelian, $x^g=t^{-1}n^{-1}xnt= x^t$. 
	
	Consider $L:= Geo_Z(\alpha(U))$ and $\tilde{L}:=\ConjGeo_Z(U)$.
	By \cite[Proposition 3.3]{[CHHR16]}, we have that  $L\subseteq X^*\cup X^*YX^*\cup X^*YX^*YX^*.$ 
	Since $U\subseteq N$, we claim that $\tilde{L}\subseteq X^*$.
	
	Notice that it is easy to see that $(X^*YX^*)\pi\ \cap N=\emptyset$ since if $(x_1yx_2)\pi\in \tilde{L}\pi\cap  (X^*YX^*)\pi \cap N$, then since $X^*\pi\subseteq N$, we would get $y\pi\in N$. Similarly, if $x_1y_1x_2y_2x_3 \in \tilde{L}\pi \cap (X^*YX^*YX^*)\pi \cap N$, then,  by property $3,$ all elements of $\Geo(G)$ are equal (in $G$) to a word with the same length and the same elements of $Y$ but all of them appearing at the end of the word, that is,  there is some $\tilde x\in X^*$ such that  $x_1y_1x_2y_2x_3 =_G \tilde{x}y_1y_2$ and $l(x_1y_1x_2y_2x_3)=l(\tilde{x}y_1y_2)$. Now, we have $(\tilde{x}y_1y_2)\pi \in \tilde{L}\pi\cap (X^*YX^*YX^*)\pi \cap N$ which implies that $(y_1y_2)\pi\in N$. By property $5$, since we have $y_1y_2=x$, for some $x\in N$ and $l(y_1y_2)=2<3$, we get $x\in X$, which contradicts the fact that $\tilde{x}y_1y_2$, and so $x_1y_1x_2y_2x_3$ is a  geodesic. Hence $\tilde{L}\subseteq X^*$.
	Notice that we just proved the inclusion $\tilde{L}\subseteq L\cap X^*$. It follows from \cite[Proposition 3.3]{[CHHR16]} that the other inclusion also holds, and so $\tilde{L}= L\cap X^*$.

	Now, we prove that $\Geo_Z(\alpha(U))\in \mc C$ (to conclude that $L\cap X^*=\tilde{L}\in \mc C$). Notice that each $a_i\in T\cup\{1\}$ defines an automorphism on $N$
	\begin{center}
		\begin{align*}
			\phi_{a_i}: N&\to N \\
			m&\mapsto m^{a_i}.
		\end{align*}
	\end{center}
	Also, since $U$ is contained in the abelian subgroup $N$, we have that
	$$
	\alpha(U)=\bigcup_{a_i\in T\cup \{1\}} U^{a_i},
	$$
	i.e,
	$$
	\alpha(U)=\bigcup_{i=0}^n \phi_{a_i}(U).
	$$
	Hence, by Lemma \ref{lema_autom}, since $\Geo_Z(U)=\Geo_X(U)\in \mc C$ and, by property 2., $\phi_{a_i}(X)=X$, for all $i$, we get $\Geo_X(\phi_{a_i}(U))=\Geo_{\phi_{a_i}(X)}(\phi_{a_i}(U))\in \mc C$, for all $i$. Further, observe that
	$$
	\Geo_X(\alpha(U))= \Geo_X(\bigcup_{i=0}^n \phi_{a_i}(U)) = \bigcup_{i=0}^n \Geo_X(\phi_{a_i}(U)),
	$$
	is in $\mc C$ (as it is the finite union of languages in $\mc C$).
	
	Therefore, $\tilde{L}\in \mc C$.
	
	Finally, recall that since $\ConjMinLenSL_Z(U)=\ConjGeo_Z(U)\cap \SL_Z(G)$, $\SL_Z(G)$ is regular by \cite[Proposition 2.5.1]{[ECHLPT92]}, and $\mc C$ is closed by intersection with regular languages, we conclude that $\ConjMinLenSL_Z(U)\in \mc C$.
\end{proof}

The general case, where the subset $U$ is not necessarily contained in the abelian subgroup, remains open. The challenges encountered in this broader setting are discussed and detailed in Section \ref{chap_open_quest}.

\begin{theorem}\label{virtab_final}
	Let $\mathcal{C}$ be a class of languages that is full semi-AFL and $G$ a virtually abelian group. We have that if $U$ is a subset of $G$ in $\mc C^{\bullet}(G)$, then $\alpha(U)$ is also a subset in $\mc C^{\bullet}(G)$.
\end{theorem}
\begin{proof}
	Let $N$ be an abelian finite index free-abelian normal subgroup of $G$.

	Let $T=\{b_1,\cdots,b_n\}$ be a transversal of $N$ in $G$ and $U\in \mathcal{C}^{\bullet}(G)$. By Theorem \ref{prop_fit_ger}, we  have
		$$
	U=\bigcup_{i=1}^n U_ib_i,\ \text{ for some } U_i\in\mathcal{C}^{\bullet}(N).
	$$
	We have that
	$$
	\alpha(U)=\bigcup_{i=1}^n \alpha(U_ib_i).
	$$
	
	Since $\mathcal{C}$ is closed for unions, so is $\mc C^{\bullet}$. Hence, to conclude the intended, it suffices to prove that
	$$
	\alpha(U_ib_i)\in \mathcal{C}^{\bullet}(G).
	$$
	For each $t\in T$ we can define an automorphism of $N$ through conjugation by $t$ as follows
	\begin{equation*}
		\begin{aligned}
			\alpha_t: N&\to N \\
			n&\mapsto t^{-1}nt
		\end{aligned}.
	\end{equation*}
	Since $N\cong \mathbb{Z}^m$, we can think of the elements of $N$ through their coordinates on the basis and view $\alpha_t$ as a matrix $Q_t$ with determinant $\pm 1$, where $u\alpha_t=uQ_t.$
	
	Now, consider $g\in G$. Then $g=us$ for some $u\in N$ and $s\in T$. Hence, given $k\in U_i$, we have
	\begin{equation*}
		\begin{aligned}
			g^{-1}kb_ig &= s^{-1}u^{-1}kb_ius = s^{-1}u^{-1}k\underbrace{b_iub_i^{-1}}_{uQ_{b_i}^{-1}}\cdot b_is \\
			&= s^{-1}u^{-1}k(uQ_{b_i}^{-1})b_is = \underbrace{s^{-1}u^{-1}k(uQ_{b_i}^{-1})s}_{(u^{-1}k(uQ_{b_i}^{-1}))Q_s}\cdot s^{-1}b_is \\
			&= (ku^{-1}\cdot uQ_{b_i}^{-1})Q_s\cdot s^{-1}b_i s  \\
			&= [k\cdot u(Q_{b_i}^{-1}-I)]Q_s\cdot s^{-1}b_is.
		\end{aligned}
	\end{equation*}
	Next, notice that in order to determine $\alpha(U_ib_i)$, we have that $k$ varies in $U_i$, $u$ varies in $N$ and $s$ varies in $T$, which implies that
	$$
	\alpha(U_ib_i)=\bigcup_{s\in T} [U_i N(Q_{b_i}^{-1}-I)]Q_s\cdot s^{-1}b_is.
	$$
	
	Observe that since $N(Q_{b_i}^{-1}-I)$ is the image of  an endomorphism of $N$, thus being a finitely generated subgroup of $N$, and so, by Theorem \ref{AnisimovSeifert}, it is rational in N. Hence, by Theorem \ref{prop_restricao}, $N(Q_{b_i}^{-1}-I)$ is a rational subset of $G$. Also, by Proposition \ref{prop_restricao}, $U_i\in \C^{\bullet}(G)$, for all $i$.  Then, by Proposition \ref{prod_rac}, $U_iN(Q_{b_i}^{-1}-I)\in \C^{\bullet}(G)$ which implies that
	$$
	[U_iN(Q_{b_i}^{-1}-I)]Q_s \in \C^{\bullet}(G),
	$$
	and so,
		$$
	[U_i.N(Q_{b_i}^{-1}-I)]Q_s \cdot s^{-1}b_is \in \C^{\bullet}(G).
	$$ 
   Finally, we get
   $$
  \alpha(U_ib_i)= \bigcup_{s\in T} [U_i N(Q_{b_i}^{-1}-I)]Q_s\cdot s^{-1}b_is \in \C^{\bullet}(G).
   $$
\end{proof}

Since $\ConjGeo(K)=\ConjGeo(G)\cap \alpha(K)\pi^{-1}$ and $\ConjGeo(G)$ is regular by \cite{[CHHR16]}, we immediately obtain the following corollary.

\begin{corollary}
	Let $\mathcal{C}$ be a class of languages that is full semi-AFL and $G$ a virtually abelian group. We have that if $K\in \mc C^{\forall}(G)$, then $\ConjGeo(K)\in \mc C$.
\end{corollary}

The previous corollary cannot be directly applied for $\mc C^\exists(G)$-subsets. Indeed, we get that $\alpha(K)\in \C^\exists(G)$, and so, since rational subsets of virtually abelian groups are closed under intersection, then the set $\ConjGeo(U)\pi=\ConjGeo(G)\pi\cap \alpha(K)$ is rational. However, that does not immediately yield that $\ConjGeo(U)=\Geo(\ConjGeo(U)\pi)$ is a regular language, as it is not true that if $K\in \Rat(G)$, then $\Geo(K)$ is regular, as happens in virtually free groups (see \cite[Corollary 3.13]{[CS24]})

\section{Questions} \label{chap_open_quest}
This paper leaves 2 main open questions regarding Lemma \ref{minepedro} and Theorem \ref{thm_va}.
\begin{enumerate}
	\item Can we prove Lemma \ref{minepedro} for the hyperbolic case? That is, can we prove that if $G$ is a hyperbolic group and $K$ is a regular language of geodesics, then there is a regular language $L_K\subseteq \Geo(\alpha(K))$ such that, for every element $g\in K$, there is at least one fully $(1,3\delta+2R+3)$-quasireduced word in $L_K$ representing an element conjugate to $g$, where $R$ is the asynchronously fellow travel constant depending on $(1,3\delta+1,\delta)$? This would allow us to generalize the main theorem of the section to rational subsets admitting a regular language of geodesic representatives.

	\item On Theorem \ref{thm_va}, we established the desired result under the restriction to  subsets of a finite index abelian subgroup. However, the natural question concerns the general case, for which we have not yet reached any concrete conclusion.
	In particular, when following the proof strategy of \cite[Proposition 3.3]{[CHHR16]}, we encounter a question that remains open for us and currently prevents us from completing the argument. Namely: is it possible for an element of $\ConjGeo(U)$ to contain a piecewise subword (i.e. not necessarily consecutive) that belongs to $\ConjGeo(G\setminus\alpha(U))$?
	A natural first step towards answering this question would be to consider concrete classes of languages, such as subsets $U$ for which $\Geo(U)$ is regular, piecewise testable, or piecewise excluding.
\end{enumerate}

\section*{Acknowledgments}
The first author was supported by CMUP (UID/00144/2025). The second author is supported by national funds through the FCT – Fundação para a Ciência e a Tecnologia, I.P., under the scope of the projects UIDB/00297/2020 (https://doi.org/10.54499/UIDB/00297/2020) and UIDP/00297/2020 (Center for Mathematics and Applications).

The authors would like to thank Laura Ciobanu for  helpful discussions, Gemma Crowe for the inspiring seminar given for the Semigroups, Automata and Languages seminars of the Center of Mathematics of the University of Porto, which led to the start of this work, and Corantin Bodart for his valuable insights that helped improve both the statement and the proof of Theorem \ref{thm_crescimento}.

\bibliographystyle{plain}
\bibliography{Bibliografia}

 \end{document}